%% file: main.tex
\documentclass[11pt]{article}
\usepackage[T1]{fontenc}
\usepackage{arxiv}  %
\usepackage{amsmath,amssymb,amsfonts,bm,booktabs}
\usepackage{algorithm}
\usepackage{algpseudocode}
\usepackage{graphicx}
\usepackage[table]{xcolor}
\usepackage[numbers,sort&compress]{natbib}
\usepackage{xspace}
\usepackage{hyperref}
\hypersetup{colorlinks=true,linkcolor=blue,citecolor=blue,urlcolor=blue}

\usepackage{amsthm}
\newtheorem{theorem}{Theorem}
\newtheorem{lemma}{Lemma}
\newtheorem{proposition}{Proposition}
\newtheorem{corollary}{Corollary}
\theoremstyle{definition}
\newtheorem{definition}{Definition}
\newtheorem{assumption}{Assumption}
\theoremstyle{remark}
\newtheorem{remark}{Remark}

\input{text/macros}

\title{A Curvature-Aware Rank-Adaptive Distributed\\
Augmented-Lagrangian Solver for Large-Scale SDPs}

\author{%
  \textbf{Hongpei Li$^{1}$}, Huikang Liu$^2$, Dongdong Ge$^2$, Yinyu Ye$^{2,3}$\\
  $^1$Northwestern University \\ 
  $^2$Shanghai Jiao Tong University \\ 
  $^3$Stanford University \\
  \texttt{HongpeiLi2031@u.northwestern.edu, hkl1u@sjtu.edu.cn}
}

\date{}

\begin{document}
\maketitle

\begin{abstract}
We present \name{} (\nameexpand), a distributed multi-GPU solver for
large-scale semidefinite programs (SDPs) based on a rank-adaptive
Burer--Monteiro factorization and an augmented Lagrangian method.  At fixed
ranks, a matrix-free L-BFGS method with negative-curvature corrections targets
an approximate Euclidean second-order stationary point of the factored
augmented Lagrangian.  A reverse multiplier shift turns a negative dual-slack
direction into exact negative curvature after rank expansion, and a small
joint rank-lift problem selects a batched low-rank correction.  A verified
slack lower bound provides an \mbox{a posteriori} approximate KKT certificate.  Our
analysis establishes generic global-optimality guarantees for heterogeneous
products of PSD cones at per-block ranks near the Barvinok--Pataki scale,
together with a finite-accuracy counterpart under blockwise cost smoothing.

For scalable execution, \name{} distributes constraint rows, factor columns,
and PSD blocks over a
\textsc{Constraint} $\times$ \textsc{Rank} $\times$ \textsc{Cone} device mesh.
The primal residual, gradient, Hessian--vector products, and slack
matrix--vector products are evaluated using device-local operations and
axis-wise collectives.  On the Mittelmann benchmark, \name{} exhibits stronger
robustness than existing low-rank GPU approaches under a uniform accuracy
standard.  Experiments on large-scale SDP relaxations from robotics,
electronic structure, and Max-Cut demonstrate the complementary scaling
regimes of the three distribution axes, with observed wall-clock speedups of
up to $4\times$ on four H100 GPUs.

\end{abstract}

\input{text/intro_new}
\input{text/theory}
\input{text/distributed_new}
\input{text/kernels}
\input{text/experiments}
\input{text/conclusion}
\begingroup
\footnotesize
\setlength{\bibsep}{0pt}
\bibliographystyle{plainnat}
\bibliography{refs}
\endgroup
\appendix
\input{text/appendix}
\end{document}

%% file: text/macros.tex
\newcommand{\name}{\textsc{CARDAL}\xspace}
\newcommand{\nameexpand}{\textbf{C}urvature-\textbf{A}ware \textbf{R}ank-Adaptive
  \textbf{D}istributed \textbf{A}ugmented \textbf{L}agrangian}

\newcommand{\R}{\mathbb{R}}                 
\newcommand{\Sym}[1]{\mathbb{S}^{#1}}       
\newcommand{\inner}[2]{\langle #1,\, #2\rangle}

\newcommand{\Aop}{\mathcal{A}}              
\newcommand{\Aadj}{\mathcal{A}^{*}}         
\newcommand{\Cobj}{C}                       
\newcommand{\bvec}{b}                       
\newcommand{\Xvar}{X}                       
\newcommand{\yvar}{y}                       
\newcommand{\Zslack}{S}                     

\newcommand{\Lrho}{\mathcal{L}_{\rho}}      

\newcommand{\Ffac}{F}                       
\newcommand{\blkdiag}{\operatorname{blkdiag}}



%% file: text/intro_new.tex
\section{Introduction}
\label{sec:introduction}

Semidefinite programming (SDP) is a fundamental modeling framework in convex
optimization \cite{Todd2001,BoydVandenberghe2004}, with applications in combinatorial optimization \cite{han2002improved,burer2001projected}, control and
robotics \cite{kang2024fast,kang2025global}, quantum information \cite{childs2007quantum,low2025fast}, and moment--sum-of-squares relaxations of
polynomial optimization problems \cite{lasserre2001global,lasserre2006convergent}.  We consider the primal--dual pair
\begin{equation}
\begin{aligned}
\textbf{(P)}\qquad
&\min_{\Xvar\in\Sym{n}}\ \inner{\Cobj}{\Xvar}
&&\text{s.t.}\quad
  \Aop\Xvar=\bvec,\quad \Xvar\succeq0, \\
\textbf{(D)}\qquad
&\max_{\yvar\in\R^m}\ \inner{\bvec}{\yvar}
&&\text{s.t.}\quad
  \Zslack(\yvar):=\Cobj-\Aadj\yvar\succeq0,
\end{aligned}
\label{eq:sdp}
\end{equation}
where $\Aop:\Sym{n}\to\R^m$ is defined by
$(\Aop\Xvar)_i=\inner{A_i}{\Xvar}$ and
$\Aadj\yvar=\sum_{i=1}^m \yvar_i A_i$.
For many emerging applications, the difficulty of \eqref{eq:sdp} is determined
jointly by the matrix order, the number of affine constraints, and the number
and heterogeneity of PSD blocks.

Interior-point solvers, such as MOSEK~\cite{aps2019mosek},
COPT~\cite{ge2022cardinal}, SDPT3~\cite{toh2012implementation},
SeDuMi~\cite{sturm1999using}, and Clarabel~\cite{goulart2026clarabel},
remain the preferred choice for high-accuracy solutions at moderate scale.
Their scalability is limited, however, by the cost of solving the Newton
systems: as the number of constraints and cone dimensions grow, fill-in and
factorization costs can dominate both memory consumption and runtime, even
when the original SDP data are sparse.

Large-scale non-interior-point approaches include bundle methods
\cite{liao2026overview,ding2023revisiting}, low-rank nonlinear programming
methods such as SDPLR \cite{BurerMonteiro2003,spdlr}, augmented-Lagrangian
and semismooth Newton methods such as SDPNAL+
\cite{zhao2010newton,yang2015sdpnal+,sun2020sdpnal+}, operator splitting such
as SCS \cite{o2016conic}, and conditional-gradient methods such as CGAL and
SketchyCGAL
\cite{yurtsever2019conditional,yurtsever2018conditional,yurtsever2021scalable}.
They avoid the dense factorizations of a full interior-point step and can
therefore address substantially larger problems, at the price of a different
accuracy--runtime tradeoff.

An important family is based on the augmented Lagrangian method (ALM).
At each ALM iteration, a subproblem for the augmented Lagrangian at penalty
parameter $\rho>0$ is solved, possibly inexactly:
\begin{equation}
\Xvar^{t+1}\ \approx\
\operatorname*{arg\,min}_{\Xvar\succeq0}\ \Lrho(\Xvar,\yvar)
=\inner{\Cobj}{\Xvar}
-\inner{\yvar}{\Aop\Xvar-\bvec}
+\frac{\rho}{2}\bigl\lVert \Aop\Xvar-\bvec\bigr\rVert_{2}^{2}
.
\label{eq:convex-alm-subproblem}
\end{equation}
Although this subproblem is convex and has only a conic constraint, solving it
globally at large scale still requires substantial spectral work to enforce
positive semidefiniteness.  Conditional-gradient methods avoid a full PSD
projection through rank-one updates, but their global rates are sublinear and
the rank of the maintained iterate can grow with the requested accuracy.

Low-rank factorization offers a different scaling regime.  Write
$m_{\mathrm{eff}}:=\operatorname{rank}(\mathcal A)$ for the number of
independent affine constraints.  For a nonempty compact feasible set, the
Barvinok--Pataki bound~\cite{Barvinok1995,Pataki1998} guarantees the existence
of an optimal SDP solution of some rank $k$ satisfying
$k(k+1)/2\leq m_{\mathrm{eff}}$.
The Burer--Monteiro (BM) substitution $X=FF^\top$
\cite{BurerMonteiro2003}, with $F\in\mathbb R^{n\times k}$ and
$k\ll n$, therefore replaces an $O(n^2)$ PSD variable by an
$O(nk)$ factor.  Applied to the convex ALM
subproblem~\eqref{eq:convex-alm-subproblem}, it gives
\begin{equation}
    F^{t+1}
    \ \approx\
    \operatorname*{arg\,min}_{F\in\mathbb R^{n\times k}}
    \Phi_{\rho,y^t}(F):=\Lrho(FF^\top,y^t).
    \label{eq:factored-alm-subproblem-intro}
\end{equation}
This reduction raises two fundamental questions: how to choose the working
rank $k$, and how to solve the resulting problem
\eqref{eq:factored-alm-subproblem-intro} reliably.  Although
\cite[Proposition~2]{ALORA} shows that the ALM subproblem
\eqref{eq:convex-alm-subproblem} admits an optimal solution satisfying the
Barvinok--Pataki rank bound under the dual Slater condition, this regularity
condition need not hold in general.  More importantly,
\eqref{eq:factored-alm-subproblem-intro} remains
nonconvex because of the quartic penalty term and may possess spurious local
minima.  Thus, the existence of a low-rank global solution does not by itself
provide an algorithmic mechanism for finding one.

A standard response is to adapt the rank until the factorized solver provides
an approximate global solution of the current convex ALM subproblem.  For a
fixed multiplier $y$, define
\[
    Z_\rho(F;y)
    :=
    \left.\nabla_X\Lrho(X,y)\right|_{X=FF^\top}.
\]
Since $\nabla_F\Phi_{\rho,y}(F)=2Z_\rho(F;y)F$,
an exact critical point satisfies $Z_\rho(F;y)F=0$.  If, in addition,
$Z_\rho(F;y)\succeq0$, then $FF^\top$ satisfies the KKT conditions of
the convex ALM subproblem and is therefore globally optimal.  If
$Z_\rho(F;y)$ has a negative eigenvalue, a corresponding eigenvector
provides an escape direction after rank lifting.  This leads to the
generic adaptive-rank inner loop
\begin{equation} \begin{aligned} F^j &\in \operatorname{crit} \left\{ F\mapsto\Lrho(FF^\top,y): F\in\mathbb R^{n\times k_j} \right\},\\ Z_\rho(F^j;y)\nsucceq0 &\quad\Longrightarrow\quad F^{j+1}=[\,F^j\ \ \alpha_jv_j\,],
         \quad v_j\approx
        v_{\min}\!\left(Z_\rho(F^j;y)\right). \end{aligned} \label{alm-framework} \end{equation}
The inner loop repeats until an inner global-accuracy contract of the
form
\begin{equation}
    \Lrho(FF^\top,y)
    -
    \inf_{X\succeq0}\Lrho(X,y)
    \leq\varepsilon
    \label{eq:convex-inner-contract}
\end{equation}
is met.  The outer iteration can then invoke standard inexact convex-ALM
theory~\cite{rockafellar1976augmented,lan2016iteration}.  This idea underlies
several prominent low-rank SDP solvers, including
SDPLR~\cite{spdlr}, ManiSDP~\cite{manisdp},
SDPDAL~\cite{wang2023decomposition}, HALLaR~\cite{HALLaR}, and
ALORA~\cite{ALORA}.

The limitation is structural: solving a transient convex ALM subproblem may
require a rank substantially larger than that of the final SDP solution.
For example, HALLaR's analysis permits $O(1/\varepsilon)$ inner correction
rounds in the worst case~\cite[Theorem~2.5]{HALLaR}; since each non-reset
Frank--Wolfe correction may add one column, the factor width can grow with
the required accuracy.  Consequently, both the $O(nk)$ storage and the cost
of the inner solve may become prohibitive.  Other methods, such as
SDPLR~\cite{spdlr} and ALORA~\cite{ALORA}, instead update the multiplier after
an approximate fixed-rank solve and a spectral rank-adaptation step, without
first certifying \eqref{eq:convex-inner-contract}.  This strategy avoids
fitting the rank to every transient convex subproblem, but existing guarantees
are local: they require the multiplier and factor initialization to be
sufficiently close to a strictly complementary primal--dual solution and
therefore do not provide a global mechanism for reaching this regime from an
arbitrary initialization.

These observations lead to the central question of this work:
\begin{center}
\emph{Under what conditions can an ALM--BM method stop rank growth near the Barvinok--Pataki scale while retaining a verifiable optimality guarantee?}
\end{center}

Another obstacle to applying this framework to substantially larger instances
is systems scalability.  Recent general-purpose low-rank GPU solvers,
including cuLoRADS~\cite{LoRADS,cuLoRADS}, cuHALLaR~\cite{cuHALLaR}, and
ALORA~\cite{ALORA}, show that matrix-free low-rank iterations map well to GPU
hardware.  Their implementations primarily target a single device, whereas the
memory footprint and computational demands of the largest graph and moment
relaxations can exceed single-device capacity.  The named-axis methodology of
D-PDLP~\cite{DPDLP}
provides a useful starting point, but a low-rank block-diagonal SDP has a third
structural dimension beyond the row--column layout of a linear program.
Distributing only the affine constraints is insufficient: work and storage
may instead be dominated by the factor columns, a few large PSD blocks, or
thousands of small blocks.

\paragraph{Contributions.} We address the optimization and systems questions
with \name{} (\nameexpand), a rank-adaptive distributed multi-GPU ALM--BM
solver.  Its two main contributions are as follows.

\begin{itemize}
\item \emph{A curvature-aware rank-adaptive ALM--BM framework with
product-cone guarantees.}
At fixed ranks, we target approximate second-order stationarity using
L-BFGS, a matrix-free negative-curvature search, and an exact quartic line
search.  A reverse multiplier shift converts a negative-slack witness
into strict negative curvature after zero-padding.  A closed-form step handles
one direction or a shared batch, while a joint rank-lift problem couples
multiple negative-slack directions across PSD blocks.  We formulate the
geometry, stationarity transfer, and convergence analysis directly for
heterogeneous products of PSD cones.  Under the stated fixed-rank stationarity
and regularity assumptions, accumulation points are second-order critical.
For almost every tuple of block costs, their represented matrices are globally
optimal once $\tau(k_c)=k_c(k_c+1)/2$ exceeds the affine dimension visible to
block $c$.
The resulting exact staircase has finite blockwise rank growth.  At finite
accuracy, slack lower bounds yield deterministic blockwise approximate KKT
certificates.  We further extend the cost-smoothed AFAC argument
of~\cite{cifuentes2022polynomial} to independently perturbed cost blocks and
transfer the result back to the nominal costs.

\item \emph{A composable three-axis multi-GPU decomposition.}
We distribute constraint rows, factor columns, and PSD blocks over a
$\textsc{Constraint}\times\textsc{Rank}\times\textsc{Cone}$ device mesh and
derive exact distributed formulas for residuals, gradients,
Hessian--vector products, and slack matrix--vector products.  These primitives
support distributed spectral searches, heterogeneous block ranks, cost-aware
cone assignment, and batching of small blocks.  The three axes address
complementary regimes and can be combined for mixed-structure instances.  The distributed
operator evaluations coincide with the product-cone operators used in the
convergence, generic-landscape, smoothed-analysis, and \mbox{a posteriori}
certification results.
\end{itemize}


\paragraph{Organization.}
Section~2 develops the product-cone BM geometry, AL-to-BM stationarity
transfer, blockwise landscape, and finite-accuracy theory.  Section~3
presents the rank-adaptive algorithm, fixed-rank convergence, exact
rank-growth analysis, and finite-output guarantees.
Section~4 gives the multi-GPU decomposition and distributed operator
identities.  Section~5 collects block operators, scaling, and implementation
safeguards.  Section~6
reports the numerical experiments, and Section~7 concludes.

%% file: text/theory.tex
\section{Theoretical Foundations for Product Cones}
\label{sec:theoretical-foundations}

CARDAL is designed for semidefinite programs with heterogeneous PSD blocks.
We therefore formulate the theory directly on a product cone and treat a
single PSD cone as the special case $q=1$.  This section answers a pointwise
question: which stationarity and curvature conditions at a factor imply exact
or approximate optimality of the original SDP?  Section~\ref{sec:algorithm-convergence}
then gives conditions under which CARDAL produces such factors.

\subsection{Product-cone BM formulation and geometry}
\label{subsec:bm-landscape}

For a positive integer $q$, write $[q]:=\{1,\ldots,q\}$.
Consider the primal--dual pair
\begin{equation}
\begin{aligned}
  \min_{\{X_c\succeq0\}_{c=1}^q}\quad
  &\sum_{c=1}^q\langle C_c,X_c\rangle,
  &\qquad \qquad  \qquad \qquad
  \max_{y\in\mathbb R^m}\quad
  &b^\top y,\\
  \text{s.t.}\quad
  &\sum_{c=1}^q\mathcal A_c(X_c)=b,
  &
  \text{s.t.}\quad
  &S_c(y):=C_c-\mathcal A_c^*(y)\succeq0
    \quad \forall c\in[q],
\end{aligned}
\label{eq:product-sdp-theory}
\end{equation}
where $\mathcal A_c:\mathbb S^{n_c}\to\mathbb R^m$ and redundant affine
rows are allowed.  We write
\[
  \mathcal A_\oplus(H_1,\ldots,H_q)
  :=\sum_{c=1}^q\mathcal A_c(H_c),
  \qquad
  m_{\mathrm{eff}}:=\operatorname{rank}(\mathcal A_\oplus).
\]
The formulation includes a conventional single-matrix SDP when $q=1$ and a
block-diagonal encoding in which the $X_c$ are the diagonal blocks.

For a rank profile
$\mathbf k=(k_1,\ldots,k_q)$, apply the Burer--Monteiro substitution
\[
  X_c=F_cF_c^\top,
  \qquad
  F_c\in\mathbb R^{n_c\times k_c},
  \qquad
  F=(F_1,\ldots,F_q).
\]
Product factor spaces use the direct-sum inner product and norm,
\[
  \langle F,U\rangle_\oplus
  :=\sum_{c=1}^q\langle F_c,U_c\rangle,
  \qquad
  \|F\|_\oplus^2:=\sum_{c=1}^q\|F_c\|_F^2.
\]
The corresponding fixed-rank BM problem is
\begin{equation*}
  \min_F\
  g(F):=\sum_{c=1}^q\langle C_c,F_cF_c^\top\rangle
  \quad\text{s.t.}\quad
  c(F):=\sum_{c=1}^q\mathcal A_c(F_cF_c^\top)-b=0,
  \tag{$\mathrm{BM}_{\mathbf k}$}
  \label{eq:bm-fixed-rank}
\end{equation*}
with feasible set
\begin{equation}
  \mathcal M_{\mathbf k}:=\{F:c(F)=0\}.
  \label{eq:product-slack-manifold}
\end{equation}
For $U=(U_c)_c$,
\begin{equation}
  Dc(F)[U]
  =\sum_{c=1}^q
    \mathcal A_c(F_cU_c^\top+U_cF_c^\top).
  \label{eq:product-jacobian}
\end{equation}
Define
\begin{equation}
  \mathcal B_F\nu
  :=\bigl(\mathcal A_c^*(\nu)F_c\bigr)_{c=1}^q,
  \qquad
  G(F):=\mathcal B_F^*\mathcal B_F.
  \label{eq:BF-and-Gram}
\end{equation}
Then $Dc(F)=2\mathcal B_F^*$.

\begin{assumption}[Product-cone BM smoothness]
\label{ass:bm-geometry}
For every rank profile $\mathbf k$ under consideration with
$\mathcal M_{\mathbf k}\neq\varnothing$, there is an open neighborhood
$\mathcal U_{\mathbf k}\supset\mathcal M_{\mathbf k}$ on which $Dc(F)$ has
constant rank.
\end{assumption}

For $q=1$, Assumption~\ref{ass:bm-geometry} is the constant-rank alternative
of the standard BM smoothness assumption
\cite[Assumption~1.1(b)]{BVB2020}.  The constant-rank theorem makes
$\mathcal M_{\mathbf k}$ an embedded submanifold, with
\begin{equation}
  T_F\mathcal M_{\mathbf k}=\ker Dc(F),
  \qquad
  N_F\mathcal M_{\mathbf k}=\operatorname{range}\mathcal B_F.
  \label{eq:tangent-normal-spaces}
\end{equation}
Moreover, $G(F)^\dagger$ and the orthogonal tangent projector
\begin{equation}
  \mathcal P_F(Z)
  :=Z-\mathcal B_FG(F)^\dagger\mathcal B_F^*Z
  =\operatorname{Proj}_{\ker Dc(F)}(Z)
  \label{eq:tangent-projector}
\end{equation}
vary smoothly on $\mathcal U_{\mathbf k}$.  Linear independence of the
factor-space constraint gradients is not required; if $Dc(F)$ has full row
rank, the compatible multiplier introduced next is unique.

For a rank staircase, nonemptiness is required at its initial profile.
Zero-padding preserves feasibility at every later profile, but smoothness
must hold at each visited profile; neither compactness nor boundedness of the
algorithmic iterates is included in Assumption~\ref{ass:bm-geometry}.

Let $C_F:=(C_cF_c)_{c=1}^q$ and define the canonical multiplier and block
slacks
\begin{equation}
  \mu_B(F):=G(F)^\dagger\mathcal B_F^*C_F,
  \qquad
  S_{B,c}(F):=C_c-\mathcal A_c^*(\mu_B(F)).
  \label{eq:canonical-multiplier}
\end{equation}
The pseudoinverse fixes a canonical representative when compatible
multipliers are not unique.  At a feasible factor, the Riemannian gradient
and Hessian quadratic form are
\begin{align}
  \operatorname{grad}g(F)
  &=\bigl(2S_{B,c}(F)F_c\bigr)_{c=1}^q,
  \label{eq:bm-rgrad}\\
  \langle U,\operatorname{Hess}g(F)[U]\rangle_\oplus
  &=2\sum_{c=1}^q\langle U_c,S_{B,c}(F)U_c\rangle,
  \qquad U\in T_F\mathcal M_{\mathbf k}.
  \label{eq:bm-rhess}
\end{align}

\begin{definition}[Product-cone BM second-order criticality]
\label{def:bm-sosp}
A factor $F\in\mathcal M_{\mathbf k}$ is a second-order critical point of
\eqref{eq:bm-fixed-rank} if
\begin{equation}
  S_{B,c}(F)F_c=0\quad(c\in[q]),
  \qquad
  \sum_{c=1}^q\langle U_c,S_{B,c}(F)U_c\rangle\geq0
  \quad\text{for every }U\in\ker Dc(F).
  \label{eq:product-bm-sosp}
\end{equation}
\end{definition}

For $y\in\mathbb R^m$ and $\rho>0$, define the product factored augmented
Lagrangian
\begin{equation}
  \Phi_{\rho,y}(F)
  :=g(F)-\langle y,c(F)\rangle+\frac{\rho}{2}\|c(F)\|_2^2,
  \qquad
  \min_F\ \Phi_{\rho,y}(F).
  \label{eq:factored-AL-subproblem}
\end{equation}
Although the augmented Lagrangian is quadratic in the matrix tuple
$(X_c)_c$, it is quartic in $F$.  Introduce the shifted multiplier
\begin{equation}
  \widehat y(F;y,\rho):=y-\rho c(F).
  \label{eq:shifted-multiplier}
\end{equation}
Direct differentiation in the product factor space gives
\begin{align}
  \nabla\Phi_{\rho,y}(F)
  &=\bigl(2S_c(\widehat y)F_c\bigr)_{c=1}^q,
  \label{eq:AL-gradient}\\
  \bigl[\nabla^2\Phi_{\rho,y}(F)[U]\bigr]_c
  &=2S_c(\widehat y)U_c
    +2\rho\mathcal A_c^*\!\left(Dc(F)[U]\right)F_c,
  \label{eq:AL-HVP}\\
  \langle U,\nabla^2\Phi_{\rho,y}(F)[U]\rangle_\oplus
  &=2\sum_{c=1}^q\langle U_c,S_c(\widehat y)U_c\rangle
    +\rho\|Dc(F)[U]\|_2^2.
  \label{eq:AL-Hessian-form}
\end{align}

\begin{proposition}[Product-cone SOSP transfer]
\label{prop:product-sosp-transfer}
Suppose Assumption~\ref{ass:bm-geometry} holds at profile $\mathbf k$.
For any $y\in\mathbb R^m$ and $\rho>0$, every feasible full-space
second-order stationary point of $\Phi_{\rho,y}$ satisfies
\eqref{eq:product-bm-sosp}.
\end{proposition}

For $q=1$, Proposition~\ref{prop:product-sosp-transfer} is exactly the
single-cone AL-to-BM transfer.  The implication is one-way: a manifold SOSP
controls curvature only on $\ker Dc(F)$, whereas the augmented Hessian acts
on the full factor space.  When the constraint gradients are dependent, a
compatible ALM multiplier need not equal $\mu_B(F)$ as a vector, but the two
slacks have the same action on $F$ and the same quadratic form on
$T_F\mathcal M_{\mathbf k}$.  This representation-invariant compatibility is
the key step in the proof.

The positive penalty term in \eqref{eq:AL-Hessian-form} also shows why exact
feasibility matters.  It vanishes on tangent directions at a feasible factor,
but can mask negative curvature at an infeasible iterate.  Finite-accuracy
statements therefore retain the feasibility and curvature residuals
explicitly.

\subsection{Generic landscape and blockwise rank conditions}
\label{subsec:product-landscape}

Let
\begin{equation}
  r_c:=\operatorname{rank}(\mathcal A_c)
      =\dim\operatorname{Im}(\mathcal A_c^*),
  \qquad
  \tau(k):=\frac{k(k+1)}2.
  \label{eq:block-affine-dimension}
\end{equation}
For one cone, the classical BM landscape result states that every
second-order critical point is globally optimal for every cost when
$k\geq n$, and for almost every cost when $k<n$ and
\begin{equation}
  \tau(k)>\operatorname{rank}(\mathcal A)
  \label{eq:bvb-rank-threshold}
\end{equation}
under the standard smoothness assumption
\cite[Proposition~3.1, Corollary~3.2, and Lemma~3.3]{BVB2020}.
Theorem~\ref{thm:product-generic-landscape} extends the single-cone conclusions
in~\cite[Theorem~1.4]{BVB2020} to heterogeneous products of PSD cones with a
separate rank condition for each block.

\begin{theorem}[Generic product-cone landscape]
\label{thm:product-generic-landscape}
Suppose Assumption~\ref{ass:bm-geometry} holds at profile $\mathbf k$ and,
for every $c\in[q]$,
\begin{equation}
  k_c\geq n_c
  \qquad\text{or}\qquad
  \tau(k_c)>r_c.
  \label{eq:product-generic-rank-condition}
\end{equation}
Then, for almost every
$C=(C_1,\ldots,C_q)\in\prod_c\mathbb S^{n_c}$, every product-cone BM
second-order critical point is globally optimal for
\eqref{eq:product-sdp-theory}.  If $k_c\geq n_c$ for every block, the
conclusion holds for every product cost.
\end{theorem}

Condition~\eqref{eq:product-generic-rank-condition} uses the dimension of the
affine directions visible to each block.  The coarser choice
$\tau(k_c)>m_{\mathrm{eff}}$ is always sufficient, whereas a condition only
on $\sum_c\tau(k_c)$ is not: columns assigned to one block cannot create
null directions with which to test the slack of another block.
Theorem~\ref{thm:product-generic-landscape} reduces to the classical
low-rank branch
\eqref{eq:bvb-rank-threshold} when $q=1$.

\subsection{Finite-accuracy certificates and cost smoothing}
\label{subsec:finite-accuracy}

At finite accuracy, the feasibility, stationarity, and curvature residuals may
be nonzero.  We use the route
\[
  \text{approximate product AL SOSP}
  \ \Longrightarrow\
  \text{product AFAC pair}
  \ \Longrightarrow\
  \text{approximate product-SDP optimality}.
\]
The first implication is deterministic.  The second is a blockwise extension
of the cost-smoothed single-cone AFAC guarantee
in~\cite[Theorem~6]{cifuentes2022polynomial}, formalized for product cones in
Theorem~\ref{thm:product-smoothed-low-rank}.

\begin{definition}[Approximate AL second-order stationarity]
\label{def:approx-sosp}
A product factor $F$ is an $(\eta,\zeta)$-approximate second-order
stationary point of \eqref{eq:factored-AL-subproblem} if
\begin{equation}
  \|\nabla\Phi_{\rho,y}(F)\|_\oplus\leq\eta,
  \qquad
  \lambda_{\min}^{\mathrm{Euc}}
  \!\left(\nabla^2\Phi_{\rho,y}(F)\right)\geq-\zeta.
  \label{eq:approx-AL-SOSP}
\end{equation}
The minimum eigenvalue is taken over the full direct-sum factor space.
\end{definition}

\begin{definition}[Approximate product-SDP optimality]
\label{def:approx-sdp-optimality}
A pair $((X_c)_c,\lambda)$ is
$(\varepsilon_0,\varepsilon_1,\varepsilon_2)$-approximately optimal for
\eqref{eq:product-sdp-theory} if
\begin{equation}
\begin{gathered}
  X_c\succeq0\quad(c\in[q]),\qquad
  \left\|\sum_c\mathcal A_c(X_c)-b\right\|_2\leq\varepsilon_0,\\
  \left(\sum_c\|S_c(\lambda)X_c\|_F^2\right)^{1/2}
    \leq\varepsilon_1,\qquad
  S_c(\lambda)\succeq-\varepsilon_2I_{n_c}\quad(c\in[q]).
\end{gathered}
\label{eq:approx-sdp-optimality}
\end{equation}
\end{definition}

\begin{proposition}[Objective interpretation of approximate KKT conditions]
\label{prop:approx-kkt-gap}
Let $((X_c)_c,\lambda)$ satisfy
\eqref{eq:approx-sdp-optimality}, and let $(X_c^\star)_c$ be any feasible
tuple.  Then
\begin{equation}
\begin{aligned}
  \sum_c\langle C_c,X_c\rangle
  -\sum_c\langle C_c,X_c^\star\rangle
  \leq{}&
  \varepsilon_0\|\lambda\|_2
  +\sqrt{\sum_c n_c}\,\varepsilon_1 +\varepsilon_2\sum_c\operatorname{tr}(X_c^\star).
\end{aligned}
\label{eq:approx-kkt-objective-gap}
\end{equation}
Thus bounded multipliers and a bounded comparison solution turn the
approximate KKT statements below into explicit one-sided objective bounds.
\end{proposition}

Whenever $\mathcal A_c$ is applied to a nonsymmetric matrix, it denotes the
extension induced by the symmetric matrices defining that operator.  Hence
\[
  Dc(F)[U]
  =2\sum_c\mathcal A_c(U_cF_c^\top).
\]

\begin{definition}[Product approximate feasible approximate criticality]
\label{def:afac}
For $\epsilon=(\epsilon_0,\epsilon_1,\epsilon_2)$ and $\gamma>0$, a pair
$(F,\lambda)$ is a product $(\epsilon,\gamma)$-AFAC pair if
\begin{align}
  \|c(F)\|_2&\leq\epsilon_0,
   \quad \left(\sum_c\|S_c(\lambda)F_c\|_F^2\right)^{1/2}
  \leq\epsilon_1,
  \label{eq:product-afac-first}\\
  \sum_c\langle U_c,S_c(\lambda)U_c\rangle
  &\geq-\epsilon_2
  \quad\text{whenever}\quad
  \|U\|_\oplus=1,\quad
  \left\|\sum_c\mathcal A_c(U_cF_c^\top)\right\|_2\leq\gamma.
  \label{eq:product-afac-second}
\end{align}
\end{definition}

\begin{proposition}[AL residuals imply product AFAC]
\label{prop:alm-to-afac}
Suppose $F$ is an $(\eta,\zeta)$-approximate product AL SOSP and set
$\lambda=\widehat y(F;y,\rho)$.  Then, for every $\gamma>0$,
$(F,\lambda)$ is a product $(\epsilon,\gamma)$-AFAC pair with
\begin{equation}
  \epsilon_0=\|c(F)\|_2,
  \qquad
  \epsilon_1=\frac{\eta}{2},
  \qquad
  \epsilon_2=\frac{\zeta}{2}+2\rho\gamma^2.
  \label{eq:AL-to-AFAC-tolerances}
\end{equation}
\end{proposition}

For the smoothed statements, fix a linearly independent representation of
the shared affine constraints.  All residuals, multiplier bounds, operator
norms, and the dimensions $r_c$ below refer to this representation.  Removing
affine redundancies leaves the primal feasible set unchanged but changes
multiplier coordinates.  Let
$\|\mathcal A_c\|_{F\to2}$ denote the operator norm from the matrix
Frobenius norm to the Euclidean norm.

\begin{theorem}[Cost-smoothed product-cone finite-accuracy guarantee]
\label{thm:product-smoothed-low-rank}
Fix $\epsilon\in\mathbb R_+^3$, $\gamma,R_F,R_\lambda>0$, and a rank
profile $\mathbf k$ before sampling the block costs.  Independently draw
\[
  C_c\sim\operatorname{Unif}
  \{D\in\mathbb S^{n_c}:\|D-\bar C_c\|_F\leq\sigma_c\}.
\]
Let
\[
  \mathcal J:=\{c:k_c\leq n_c,\quad
                      a_c:=\|\mathcal A_c\|_{F\to2}>0\},
  \qquad
  \delta_c:=\frac{\epsilon_1a_c}{\gamma},
  \qquad
  \kappa_c:=R_\lambda a_c.
\]
Suppose, for every $c\in\mathcal J$,
\begin{equation}
  \tau(k_c)>r_c,
  \qquad
  0<\delta_c\leq\kappa_c,
  \qquad
  \delta_c<\frac{\sigma_c}{4n_c^3}.
  \label{eq:product-smoothed-basic-conditions}
\end{equation}
Then, with probability at least $1-\sum_{c\in\mathcal J}p_c$, where
\begin{equation}
  p_c
  :=4\mathrm e
    \left(\frac{3\kappa_c}{\delta_c}\right)^{r_c}
    \left(\frac{4n_c^3\delta_c}{\sigma_c}\right)^{\tau(k_c)},
  \label{eq:product-smoothed-failure}
\end{equation}
every product $(\epsilon,\gamma)$-AFAC pair satisfying
$\|F\|_\oplus\leq R_F$ and $\|\lambda\|_2\leq R_\lambda$ obeys
\begin{equation}
\begin{gathered}
  X_c:=F_cF_c^\top\succeq0,\qquad
  \left\|\sum_c\mathcal A_c(X_c)-b\right\|_2\leq\epsilon_0,\\
  \left(\sum_c\|S_c(\lambda)X_c\|_F^2\right)^{1/2}
    \leq R_F\epsilon_1,\qquad
  S_c(\lambda)\succeq-\epsilon_2I_{n_c}\quad(c\in[q]).
\end{gathered}
\label{eq:product-smoothed-kkt}
\end{equation}
Blocks with $k_c>n_c$ or $a_c=0$ satisfy the final inequality
deterministically and do not contribute a failure probability.
\end{theorem}

The familiar near-Barvinok--Pataki scale follows blockwise.  If, for fixed
$\chi,t_c>0$,
\begin{equation}
  \tau(k_c)\geq(1+\chi)r_c+\chi t_c,
  \qquad
  \delta_c\leq
  \left(\frac1{3\kappa_c}\right)^{1/\chi}
  \left(\frac{\sigma_c}{4n_c^3}\right)^{1+1/\chi},
  \qquad
  \frac{\sigma_c}{12\kappa_cn_c^3}<1,
  \label{eq:product-smoothed-simple-conditions}
\end{equation}
then
\[
  p_c\leq
  4\mathrm e
  \left(\frac{\sigma_c}{12\kappa_cn_c^3}\right)^{t_c}.
\]
For fixed $\chi$ and $t_c$, the condition is
$k_c\sim\sqrt{2(1+\chi)r_c}$, rather than a full-width factor.  Taking
$q=1$ recovers the single-cone cost-smoothed AFAC guarantee.

\begin{remark}[Scope of the smoothed guarantee]
\label{rem:smoothed-scope}
Theorem~\ref{thm:product-smoothed-low-rank} applies when the sampled block
costs are used throughout the run.  Its event is uniform over all bounded
AFAC pairs at the prescribed rank profile, which is essential because the output
depends on the sampled costs.  It does not describe a run performed only at
the nominal costs; Section~\ref{subsec:rank-growth-finite-guarantees}
transfers the sampled-cost output back to those costs.  We use the
AFAC-to-optimality argument of \cite{cifuentes2022polynomial}, but do not
claim its polynomial iteration bound for the implemented ALM--L-BFGS--NC
procedure.  A data-dependent profile or tolerance envelope requires a finite
union bound or a separate covering argument.
\end{remark}

\section{CARDAL Algorithm and Convergence Guarantees}
\label{sec:algorithm-convergence}

We now address the algorithmic question: under what conditions does CARDAL
produce the factors characterized in Section~\ref{sec:theoretical-foundations},
and when does blockwise rank growth stop?  The outer method applies ALM directly
to the nonconvex product BM problem \eqref{eq:bm-fixed-rank}.  With the rank
profile fixed, its inner routine targets an approximate full-space SOSP of
\eqref{eq:factored-AL-subproblem}; the rank-adaptive wrapper then tests each
dual-slack block and enlarges only the blocks with detected negative
curvature.

\input{text/algorithm}

\subsection{Fixed-rank ALM convergence}
\label{subsec:fixed-rank-convergence}

Consider a tail on which the rank profile $\mathbf k$ is fixed.  For
the asymptotic analysis, suppress the finite residual stopping test.  Given
$(F_\ell,y_\ell,\rho_\ell)$, let $F_{\ell+1}$ be the accepted output of the
inner solve with base multiplier $y_\ell$, and set
\begin{equation}
  y_{\ell+1}=y_\ell-\rho_\ell c(F_{\ell+1}),
  \qquad
  \delta_\ell:=\|c(F_\ell)\|_2.
  \label{eq:fixed-rank-indexing}
\end{equation}

\begin{assumption}[Fixed-rank ALM conditions]
\label{ass:fixed-rank-alm}
Fix a rank profile $\mathbf k$ and suppose that:
\begin{enumerate}
  \item[(A1)] $\mathcal M_{\mathbf k}$ is nonempty, the SDP feasible set is
  compact, and the accepted factors and relevant inner level sets lie in a
  compact set;
  \item[(A2)] Assumption~\ref{ass:bm-geometry} holds at $\mathbf k$;
  \item[(A3)] every $F_{\ell+1}$ satisfies
  \eqref{eq:inner-sosp-contract} for $(y_\ell,\rho_\ell)$, with
  $\omega_\ell\to0$ and $\zeta_\ell\to0$;
  \item[(A4)] $\rho_0>0$, the multipliers $\{y_\ell\}$ are bounded, and the
  uncapped residual-shrinkage rule \eqref{eq:penalty-update} is used.
\end{enumerate}
\end{assumption}

\begin{lemma}[Vanishing primal residual]
\label{lem:vanishing-residual}
Under Assumption~\ref{ass:fixed-rank-alm},
\begin{equation}
  \delta_\ell\longrightarrow0.
  \label{eq:residual-vanishes}
\end{equation}
\end{lemma}

\begin{proof}
If $\rho_\ell$ is eventually constant, successful residual tests eventually
give $\delta_{\ell+1}\leq\theta\delta_\ell$.  Otherwise
$\rho_\ell\to\infty$ along the penalty-increase indices, and boundedness of
the multipliers in \eqref{eq:fixed-rank-indexing} gives
\[
  \delta_{\ell+1}
  \leq\frac{2\sup_j\|y_j\|_2}{\rho_\ell}\longrightarrow0
\]
along those indices; every intervening successful test contracts the
residual by at least $\theta$.  Hence the full sequence converges to zero.
\end{proof}

\begin{remark}[Finite penalties in the local regime]
\label{rem:finite-local-penalty}
The diverging-penalty branch above is a proof device, not a prediction.
Standard local ALM theory permits a sufficiently large finite penalty to
remain constant when an exact local minimizer branch satisfies the usual
regularity and second-order sufficient conditions
\cite[Proposition~2.7]{Bertsekas1982}.  That result explains practical
penalty stabilization but does not by itself prove local convergence of the
inexact L-BFGS--NC inner solver.
\end{remark}

\begin{theorem}[Fixed-rank product-cone ALM convergence]
\label{thm:product-fixed-rank-sosp}
Under Assumption~\ref{ass:fixed-rank-alm}, every accumulation point of the
accepted factor sequence is feasible and satisfies
\eqref{eq:product-bm-sosp}.  If, in addition, the blockwise rank conditions
\eqref{eq:product-generic-rank-condition} hold and the product cost lies
outside the exceptional set of
Theorem~\ref{thm:product-generic-landscape}, then every accumulation point is
globally optimal for \eqref{eq:product-sdp-theory},
\begin{equation}
  g(F_\ell)\to p_\star,
  \qquad
  \operatorname{dist}\!\left(
    (F_{\ell,c}F_{\ell,c}^\top)_c,\mathcal X^\star
  \right)\to0,
  \label{eq:matrix-sequence-to-solution-set}
\end{equation}
where $\mathcal X^\star$ is the product-SDP solution set.  If
$k_c\geq n_c$ for every block, the conclusion holds for every product cost.
\end{theorem}

The single-cone result is recovered when $q=1$.  Because each factor is
invariant under $F_c\mapsto F_cQ_c$, the theorem is stated in terms of
accumulation points and represented matrices; convergence of an entire factor
sequence would require additional quotient-manifold or finite-length
assumptions.

\subsection{Blockwise rank growth and finite outputs}
\label{subsec:rank-growth-finite-guarantees}

We separate three conclusions.  An idealized exact staircase explains finite
blockwise rank growth, an \mbox{a posteriori} slack test certifies any finite output
meeting the stated residual bounds, and cost smoothing gives a low-rank
finite-accuracy statement.  None of these conclusions bounds the numerical
work at a given rank.

Define the first certifying rank of block $c$ by
\begin{equation}
  k_c^{\mathrm{gen}}
  :=\min\{k\geq k_{c,0}: k\geq n_c\ \text{or}\ \tau(k)>r_c\}.
  \label{eq:product-first-generic-rank}
\end{equation}
An \emph{exact product staircase} solves each visited profile to a feasible
full-space AL SOSP, tests every slack block with zero tolerance, and enlarges
precisely the blocks with a negative slack eigenvalue.

\begin{corollary}[Finite exact blockwise rank growth]
\label{cor:product-exact-rank-growth}
Suppose Assumption~\ref{ass:bm-geometry} holds at every visited profile up to
$\mathbf k^{\mathrm{gen}}$, and exclude the finite union of the exceptional
cost sets from Theorem~\ref{thm:product-generic-landscape} over those
profiles.  Then no block in an exact product staircase is enlarged beyond
$k_c^{\mathrm{gen}}$.  If block $c$ receives
$r_{\mathrm{inc},c}$ columns per update, with the final update truncated at
$k_c^{\mathrm{gen}}$, the total number of accepted block updates is at most
\begin{equation}
  \sum_{c=1}^q
  \left\lceil
    \frac{k_c^{\mathrm{gen}}-k_{c,0}}{r_{\mathrm{inc},c}}
  \right\rceil.
  \label{eq:product-rank-update-bound}
\end{equation}
The terminal tuple satisfies the primal--dual KKT conditions and is globally
optimal.
\end{corollary}

This exact statement concerns staircase stages; it does not assert finite time
to reach an exact SOSP.  More generally, any infinite run with a bounded
integer rank profile has a fixed-rank tail after its final accepted update, so
Theorem~\ref{thm:product-fixed-rank-sosp} applies if
Assumption~\ref{ass:fixed-rank-alm} holds there.  Reaching a rank cap alone is
not a certificate: every block at its rank cap must still pass the slack test.

\begin{proposition}[Finite-output product approximate KKT certificate]
\label{prop:finite-output-kkt}
Let $X_c=F_cF_c^\top$ and let
$y_{\mathrm{out}}=y_{\mathrm{base}}-\rho c(F)$ be the multiplier returned
after a fixed-rank ALM subproblem.  Suppose
\begin{equation}
  \|c(F)\|_2\leq\delta_{\mathrm{feas}},
  \qquad
  \|\nabla\Phi_{\rho,y_{\mathrm{base}}}(F)\|_\oplus\leq\omega,
  \qquad
  S_c(y_{\mathrm{out}})\succeq-\tau_{\mathrm{dual}}I_{n_c}\ (c\in[q]),
  \qquad
  \|F\|_\oplus\leq R_F.
  \label{eq:finite-kkt-residuals}
\end{equation}
Then $((X_c)_c,y_{\mathrm{out}})$ is
\begin{equation}
  \left(
    \delta_{\mathrm{feas}},
    \frac{R_F\omega}{2},
    \tau_{\mathrm{dual}}
  \right)
  \text{-approximately optimal}
  \label{eq:finite-kkt-tolerances}
\end{equation}
in the sense of Definition~\ref{def:approx-sdp-optimality}.
\end{proposition}

This deterministic certificate does not invoke a generic landscape theorem:
the blockwise slack tests directly supply approximate dual feasibility.
Together with Proposition~\ref{prop:approx-kkt-gap}, it also yields a
one-sided objective bound whenever the multiplier and a comparison solution
are bounded.

For the smoothed output statement, choose deterministic envelopes before
sampling:
\[
  \bar\delta,\ \bar\omega,\ \bar\zeta,\ \bar\rho,\
  \gamma,\ R_F,\ R_y>0
\]
and set
\begin{equation}
  \epsilon_0:=\bar\delta,
  \qquad
  \epsilon_1:=\frac{\bar\omega}{2},
  \qquad
  \epsilon_2:=\frac{\bar\zeta}{2}+2\bar\rho\gamma^2.
  \label{eq:fixed-smoothed-tolerances}
\end{equation}

\begin{corollary}[Cost-smoothed product-cone ALM output]
\label{cor:product-smoothed-output}
Suppose the independently sampled block costs of
Theorem~\ref{thm:product-smoothed-low-rank} are used throughout the run and,
at a rank profile chosen before sampling, the terminal output satisfies
\begin{equation}
  \|c(F)\|_2\leq\bar\delta,
  \qquad
  \|\nabla\Phi_{\rho,y}(F)\|_\oplus\leq\bar\omega,
  \qquad
  \lambda_{\min}^{\mathrm{Euc}}
    \!\left(\nabla^2\Phi_{\rho,y}(F)\right)\geq-\bar\zeta,
  \qquad
  \rho\leq\bar\rho,
  \label{eq:smoothed-output-residuals}
\end{equation}
and, with $\widehat y=y-\rho c(F)$,
\begin{equation}
  \|F\|_\oplus\leq R_F,
  \qquad
  \|\widehat y\|_2\leq R_y.
  \label{eq:smoothed-output-bounds}
\end{equation}
Under the conditions of
Theorem~\ref{thm:product-smoothed-low-rank}, with $R_\lambda=R_y$, the
sampled-cost output satisfies \eqref{eq:product-smoothed-kkt} with probability
at least $1-\sum_{c\in\mathcal J}p_c$.  Let
$\sigma_{\max}:=\max_c\sigma_c$.  For the nominal costs $(\bar C_c)_c$, the
same output is approximately optimal with tolerances
\begin{equation}
  \left(
    \epsilon_0,
    R_F\epsilon_1+\sigma_{\max}R_F^2,
    \epsilon_2+\sigma_{\max}
  \right).
  \label{eq:product-nominal-kkt}
\end{equation}
If $\tau_{\mathrm{dual}}\geq\epsilon_2$, no block covered by the theorem can
produce a sampled-slack direction below $-\tau_{\mathrm{dual}}$ on the same
event.
\end{corollary}

The additive $\sigma_{\max}$ terms are the accuracy floor incurred when
transferring the sampled-cost result to the nominal problem.  For a terminal
profile chosen from a finite predeclared set, apply a union bound to
Theorem~\ref{thm:product-smoothed-low-rank} over the eligible profiles.  No
additional union over iterates is needed because the event for each prescribed
rank profile is uniform over bounded AFAC pairs.

%% file: text/algorithm.tex
\subsection{Rank-adaptive CARDAL}
\label{subsec:alm-lbfgs-nc}

We organize \name{} into three nested routines.  The rank-adaptive wrapper in
Algorithm~\ref{alg:cardal-core} is the only routine that changes the
rank profile $\mathbf k=(k_1,\ldots,k_q)$.  At each profile, it calls
the fixed-rank ALM in Algorithm~\ref{alg:fixed-rank-alm}, whose
subproblems are solved by the L-BFGS--NC method in
Algorithm~\ref{alg:lbfgs-nc}.  Because all blocks enter the shared residual
$c(F)$, they use one multiplier and one penalty and are optimized jointly in
the direct-sum factor space.  Throughout Algorithms~\ref{alg:fixed-rank-alm}
and~\ref{alg:lbfgs-nc}, $F$ denotes the tuple $(F_c)_{c=1}^q$, rather than a
single common-size matrix.  Its entries may be assembled into a single
optimization vector, but the objective, operators, gradients,
and search directions retain their blockwise definitions.  Only
Algorithm~\ref{alg:cardal-core} changes the widths $k_c$ of these blocks.
Here \emph{fixed-rank} means that every entry of $\mathbf k$ remains unchanged
during the ALM call; the block ranks need not be equal.

For fixed $(y,\rho,\mathbf k)$, the inner solver targets
\begin{equation}
  \|\nabla\Phi_{\rho,y}(F)\|_\oplus\leq\omega,
  \qquad
  \lambda_{\min}^{\mathrm{Euc}}
  \bigl(\nabla^2\Phi_{\rho,y}(F)\bigr)\geq-\zeta.
  \label{eq:inner-sosp-contract}
\end{equation}

\paragraph{Rank-adaptive wrapper.}
Algorithm~\ref{alg:cardal-core} first runs the fixed-rank ALM and then
tests every returned slack block.  For block $c$, the algorithm retains at most
$r_{\mathrm{inc},c}$ directions below $-\tau_{\mathrm{dual}}$.  If no block
violates the test, the wrapper terminates; otherwise, one joint rank-lift
problem couples all retained directions through the shared affine residual.

\begin{algorithm}[H]
\caption{\name{} rank-adaptive wrapper}
\label{alg:cardal-core}
\begin{algorithmic}[1]
\State \textbf{Input:} product factor $F=(F_c)_c$, profile $\mathbf k$,
multiplier $y_{\mathrm{base}}$, penalty $\rho>0$, slack tolerance
$\tau_{\mathrm{dual}}$, and increment caps $(r_{\mathrm{inc},c})_c$
\Loop
  \State $(F,y,\rho)\gets
    \Call{FixedRankALM}{F,y_{\mathrm{base}},\rho}$
  \For{$c=1,\ldots,q$}
    \State $\mathcal E_c\gets
      \Call{NegativeEigenpairs}
      {S_c(y),\tau_{\mathrm{dual}},r_{\mathrm{inc},c}}$
  \EndFor
  \If{$\mathcal E_c=\varnothing$ for every $c$}
    \State \Return $(F,y)$
  \Else
    \State $y_{\mathrm{base}}\gets y + \rho c(F)$
    \State $(W_c)_c\gets
      \Call{JointRankLift}{y,\rho,(\mathcal E_c,\mathcal A_c)_c}$
    \ForAll{$c$ with $W_c\neq\varnothing$}
      \State $F_c\gets[\,F_c\ \ W_c\,]$ and
      $k_c\gets k_c+\operatorname{cols}(W_c)$
    \EndFor
  \EndIf
\EndLoop
\end{algorithmic}
\end{algorithm}

The following single-cone lemma isolates the reverse-shift mechanism and its
one-dimensional quartic line search.  The product update then applies the
same identity simultaneously to the selected blocks.  To ensure that these
rank-lift constructions are well posed, we assume that the primal feasible
set is nonempty and compact.  This compactness condition also appears in
Assumption~\ref{ass:fixed-rank-alm} and rules out nonzero PSD recession
directions in the nullspace of the affine map.

\begin{lemma}[Reverse-shift rank lift and exact line search]
\label{lem:reverse-shift-rank-lift}
Let $F\in\mathbb R^{n\times k}$ be produced by a subproblem with base
multiplier $y_{\mathrm{base}}$ and penalty $\rho$, and set
\[
  y_{\mathrm{out}}=y_{\mathrm{base}}-\rho c(F).
\]
For any unit vector $v$, set
$q_v:=v^\top S(y_{\mathrm{out}})v$ and suppose $q_v<0$.  Define
\[
  \bar F=[\,F\ \ 0\,],
  \qquad
  U_v=[\,0\ \ v\,],
  \qquad
  a_v=\mathcal A(vv^\top).
\]
Then
\begin{equation}
  Dc(\bar F)[U_v]=0,
  \qquad
  \left\langle
    U_v,\nabla^2\Phi_{\rho,y_{\mathrm{base}}}(\bar F)[U_v]
  \right\rangle
  =2q_v<0.
  \label{eq:reverse-shift-rank-lift}
\end{equation}
Moreover, the enlarged-rank objective satisfies
\begin{equation}
  \Phi_{\rho,y_{\mathrm{base}}}(\bar F+tU_v)
  =
  \Phi_{\rho,y_{\mathrm{base}}}(\bar F)
  +q_vt^2+\frac{\rho}{2}\|a_v\|_2^2t^4.
  \label{eq:rank-lift-line-restriction}
\end{equation}
If $a_v\neq0$, its exact nonnegative minimizer is
\begin{equation}
  t_\star
  =
  \sqrt{\frac{-q_v}{\rho\|a_v\|_2^2}}.
  \label{eq:single-cone-rank-lift-step}
\end{equation}
\end{lemma}

\begin{proof}
The added column is orthogonal to the zero-padded factor in the sense that
$\bar F U_v^\top+U_v\bar F^\top=0$.  Hence
$Dc(\bar F)[U_v]=0$.  Moreover, $c(\bar F)=c(F)$, so
\[
  \widehat y(\bar F;y_{\mathrm{base}},\rho)
  =y_{\mathrm{base}}-\rho c(F)
  =y_{\mathrm{out}}.
\]
This proves \eqref{eq:reverse-shift-rank-lift} using
\eqref{eq:AL-Hessian-form}.  For any $t\in\mathbb R$,
\[
  c(\bar F+tU_v)=c(F)+t^2a_v,
  \qquad
  g(\bar F+tU_v)=g(F)+t^2v^\top Cv.
\]
Expanding the augmented Lagrangian and using
$y_{\mathrm{out}}=y_{\mathrm{base}}-\rho c(F)$ gives
\eqref{eq:rank-lift-line-restriction}.  Differentiating its even quartic
polynomial gives \eqref{eq:single-cone-rank-lift-step}.
\end{proof}

Compactness ensures $a_v\neq0$ for nonzero $v$:
otherwise $X+s vv^\top$ would be feasible for every $s\geq0$ whenever $X$ is
feasible.  Hence \eqref{eq:single-cone-rank-lift-step} is well defined and
eliminates the rank-growth scale hyperparameter.
The $y_{\mathrm{base}}$ assignment in Algorithm~\ref{alg:cardal-core}
recovers the base multiplier for which this line search is exact.

\paragraph{Joint rank lift.}
Let $\mathcal J$ be the set of violated blocks.  For each $c\in\mathcal J$,
collect the retained orthonormal directions in
$V_c=[\,v_{c,1}\ \cdots\ v_{c,\ell_c}\,]$ and define
\[
  H_c:=V_c^\top S_c(y_{\mathrm{out}})V_c,
  \qquad
  q_{c,j}:=v_{c,j}^\top S_c(y_{\mathrm{out}})v_{c,j}<0,
  \qquad
  a_{c,j}:=\mathcal A_c(v_{c,j}v_{c,j}^\top).
\]
For $Z_c\in\mathbb S_+^{\ell_c}$, set
$F_c(Z_c)=[\,F_c\ \ V_cZ_c^{1/2}\,]$ on selected blocks and leave the
remaining factors unchanged.  Lemma~\ref{lem:reverse-shift-rank-lift} gives
\begin{equation}
\begin{aligned}
  \Phi_{\rho,y_{\mathrm{base}}}(F(Z))
  ={}&\Phi_{\rho,y_{\mathrm{base}}}(F)
  +\sum_{c\in\mathcal J}\langle H_c,Z_c\rangle\\
  &+\frac{\rho}{2}
    \left\|
      \sum_{c\in\mathcal J}
        \mathcal A_c(V_cZ_cV_c^\top)
    \right\|_2^2.
\end{aligned}
  \label{eq:joint-rank-lift-restriction}
\end{equation}
Thus all selected cones and directions are coupled through the convex
quadratic SDP
\begin{equation}
  (Z_c^\star)_{c\in\mathcal J}
  \in
  \arg\min_{Z_c\succeq0,\ c\in\mathcal J}
  \left\{
    \sum_{c\in\mathcal J}\langle H_c,Z_c\rangle
    +\frac{\rho}{2}
      \left\|
        \sum_{c\in\mathcal J}
          \mathcal A_c(V_cZ_cV_c^\top)
      \right\|_2^2
  \right\}.
  \label{eq:product-joint-rank-lift-step}
\end{equation}
The PSD variables mix directions only within their own cone, while the
shared residual couples the cones.  The same recession argument makes the
objective coercive on the selected PSD subspaces.  The correction
$W_c=V_c(Z_c^\star)^{1/2}$ is compressed to the positive eigenspace of
$Z_c^\star$ before columns are allocated.

Restricting each matrix to
$Z_c=\operatorname{Diag}(z_{c,1},\ldots,z_{c,\ell_c})$ gives the joint
nonnegative QP
\begin{equation}
  \min_{z\geq0}
  \left\{
    \sum_{c\in\mathcal J}\sum_{j=1}^{\ell_c}q_{c,j}z_{c,j}
    +\frac{\rho}{2}
      \left\|
        \sum_{c\in\mathcal J}\sum_{j=1}^{\ell_c}z_{c,j}a_{c,j}
      \right\|_2^2
  \right\}.
  \label{eq:diagonal-rank-lift-qp}
\end{equation}
Restricting further to $Z_c=t^2I_{\ell_c}$ gives the shared closed-form
amplitude
\begin{equation}
  t_\star
  =
  \sqrt{
    \frac{-\sum_{c\in\mathcal J}\sum_{j=1}^{\ell_c}q_{c,j}}
    {\rho\left\|
      \sum_{c\in\mathcal J}\sum_{j=1}^{\ell_c}a_{c,j}
    \right\|_2^2}
  }.
  \label{eq:product-common-rank-lift-step}
\end{equation}
For one selected direction in one cone, this reduces to
\eqref{eq:single-cone-rank-lift-step}.

\paragraph{Fixed-rank ALM.}
With $\mathbf k$ fixed, the method alternates between an approximate subproblem
solve and the multiplier update $y^+=y-\rho c(F^+)$.  The penalty is
increased only when the primal residual fails to contract:
\begin{equation}
  \rho^+
  =
  \begin{cases}
    \rho,
      & \|c(F^+)\|_2\leq\theta\|c(F)\|_2,\\
    \gamma_\rho\rho,
      & \text{otherwise},
  \end{cases}
  \qquad
  0<\theta<1,
  \quad
  \gamma_\rho>1.
  \label{eq:penalty-update}
\end{equation}

\begin{algorithm}[H]
\caption{Fixed-rank augmented-Lagrangian method}
\label{alg:fixed-rank-alm}
\begin{algorithmic}[1]
\State \textbf{Input:} product factor $F=(F_c)_{c=1}^q$, profile $\mathbf k$,
multiplier $y$, and initial penalty $\rho>0$
\Loop
  \State $\delta\gets\|c(F)\|_2$ and choose inner tolerances
    $(\omega,\zeta)$
  \State $F\gets\Call{L-BFGS--NC}{F,y,\rho,\omega,\zeta}$
  \State $r\gets c(F)$ and $y\gets y-\rho r$
  \If{$\|r\|_2\leq\epsilon_0$}
    \State \Return $(F,y,\rho)$
  \EndIf
  \If{$\|r\|_2>\theta\delta$}
    \State $\rho\gets\gamma_\rho\rho$
  \EndIf
\EndLoop
\end{algorithmic}
\end{algorithm}

The fixed-rank routine neither computes a slack eigenpair nor changes
$\mathbf k$.  The convergence analysis assumes that the current profile can
represent a feasible point.  All rank updates are determined by the
outer dual-slack test in Algorithm~\ref{alg:cardal-core}.

\paragraph{L-BFGS--NC subproblem solver.}
For fixed $(y,\rho,\mathbf k)$, Algorithm~\ref{alg:lbfgs-nc} first reduces the
gradient using safeguarded L-BFGS steps.  Once the gradient tolerance is met,
a matrix-free spectral test is applied to the true Hessian.  If negative
curvature is detected, the method takes a curvature step and resets the
L-BFGS memory; otherwise it returns when the second-order tolerance is met.
To limit the cost of Hessian--vector products, the practical implementation
invokes the negative-curvature search only intermittently after sufficient
first-order progress, while the analysis uses the second-order termination
condition above.

\begin{algorithm}[H]
\caption{L-BFGS--NC for one augmented-Lagrangian subproblem}
\label{alg:lbfgs-nc}
\begin{algorithmic}[1]
\State \textbf{Input:} product factor $F=(F_c)_{c=1}^q$, profile $\mathbf k$,
multiplier $y$, penalty $\rho$, tolerances $(\omega,\zeta)$, and maximum step
$\tau_{\max}$
\Loop
  \While{$\|\nabla\Phi_{\rho,y}(F)\|_\oplus>\omega$}
    \State Compute a safeguarded L-BFGS direction $D=(D_c)_{c=1}^q$
    \State $\tau_\star\gets
      \arg\min_{0\leq\tau\leq\tau_{\max}}
      \Phi_{\rho,y}(F+\tau D)$
    \State $F\gets F+\tau_\star D$ and update the L-BFGS memory
  \EndWhile
  \If{$\lambda_{\min}^{\mathrm{Euc}}
      (\nabla^2\Phi_{\rho,y}(F))\geq-\zeta$}
    \State \Return $F$
  \EndIf
  \State Let $V=(V_c)_{c=1}^q$ be the returned unit
    negative-curvature direction
  \State $\tau_\star\gets
    \arg\min_{|\tau|\leq\tau_{\max}}
    \Phi_{\rho,y}(F+\tau V)$
  \State $F\gets F+\tau_\star V$ and clear the L-BFGS memory
\EndLoop
\end{algorithmic}
\end{algorithm}

Along either search direction, the augmented-Lagrangian objective is a
quartic polynomial in the step length, so the exact line search requires only
the interval endpoints and the admissible real roots of its cubic derivative.
The Hessian test in Algorithm~\ref{alg:lbfgs-nc} targets second-order
stationarity at the current profile, whereas the blockwise slack tests in
Algorithm~\ref{alg:cardal-core} determine which ranks should increase.

%% file: text/distributed_new.tex

\section{Distributed Multi-GPU Parallelization}
\label{sec:distributed-multigpu}

The difficulty of scaling a semidefinite program is not captured by its
ambient matrix dimension alone.  Large application instances exhibit at
least three distinct forms of structural complexity.  Electronic-structure
and orbital-relaxation models can contain very large families of
affine constraints, making the applications of $\mathcal A$ and
$\mathcal A^*$ the dominant cost.  Some difficult Max-Cut relaxations require
a wide working factor before a rank-adaptive method passes its terminal
dual-slack test; factor storage,
slack--factor products, and Hessian--vector products then become the
bottleneck.
Finally, sparse sum-of-squares and moment relaxations naturally produce many
PSD cones, often with highly nonuniform dimensions and sparsity patterns.
These regimes are not mutually exclusive: a single instance may have many
constraints, wide factors, and many heterogeneous cones.

This structural diversity motivates a three-axis distribution rather than a
one-dimensional data split.  We partition the affine constraints, the factor
columns, and the PSD cones over a
$\textsc{Constraint}\times\textsc{Rank}\times\textsc{Cone}$ device mesh.  The
resulting layout is an exact decomposition of the factorized SDP, not an
approximation.  Moreover, all global quantities required by the ALM--BM
method can be evaluated through two complementary collective patterns: a
forward reduction over the Rank and Cone axes, and an adjoint reduction over
the Constraint axis.

\subsection{Structural regimes and a three-axis decomposition}
\label{subsec:three-axis-decomposition}

Consider the block-diagonal form of the primal SDP,
\begin{equation}
\begin{aligned}
    \min_{\{X_c\succeq 0\}_{c=1}^{q}}\quad
    & \sum_{c=1}^{q} \langle C_c,X_c\rangle,\\
    \text{s.t.}\quad
    & \sum_{c=1}^{q}\langle A_{c,i},X_c\rangle=b_i,
      \qquad i=1,\ldots,m,
\end{aligned}
\label{eq:block-sdp-distributed}
\end{equation}
where $A_{c,i}$ is the restriction of the $i$th constraint matrix to
cone $c$.  For clarity, we omit the optional nonnegative component; it can be
represented by $1\times1$ PSD cones and distributed in the same way.  Let
\begin{equation}
    [m]=\bigsqcup_{p=1}^{P_R} I_p,
    \qquad
    [q]=\bigsqcup_{\ell=1}^{P_C}\Lambda_\ell
\label{eq:constraint-cone-partitions}
\end{equation}
be partitions of the constraints and cones.  For each cone, partition the
columns of its factor as
\begin{equation}
    [k_c]=\bigsqcup_{s=1}^{P_K}K_{c,s},
    \qquad
    F_c=\bigl[F_c^{[1]}\ \cdots\ F_c^{[P_K]}\bigr],
    \qquad
    F_c^{[s]}:=F_c[:,K_{c,s}].
\label{eq:rank-partition}
\end{equation}
Empty $K_{c,s}$ are allowed when $k_c<P_K$.  Column partitioning is
particularly convenient because it introduces no cross terms:
\begin{equation}
    X_c=F_cF_c^\top
       =\sum_{s=1}^{P_K}F_c^{[s]}F_c^{[s]\top}.
\label{eq:rank-separable-outer-product}
\end{equation}

The subscripts $R$, $K$, and $C$ denote the Constraint (row), Rank, and Cone
mesh axes, respectively; $R$ is therefore not the factor rank.

For $i\in I_p$, define the constraint--cone block operator
\begin{equation}
    \bigl[\mathcal A_{p\ell}(X_{\Lambda_\ell})\bigr]_i
    :=\sum_{c\in\Lambda_\ell}\langle A_{c,i},X_c\rangle.
\label{eq:block-operator-p-ell}
\end{equation}
Substituting \eqref{eq:rank-separable-outer-product} into
\eqref{eq:block-sdp-distributed} gives the fully partitioned BM problem
\begin{equation}
\begin{aligned}
    \min_{\{F_c^{[s]}\}_{c\in[q],\,s\in[P_K]}}\quad
    &\sum_{\ell=1}^{P_C}\sum_{s=1}^{P_K}
      \sum_{c\in\Lambda_\ell}
      \left\langle C_c,F_c^{[s]}F_c^{[s]\top}\right\rangle,\\
    \text{s.t.}\quad
    &\sum_{\ell=1}^{P_C}\sum_{s=1}^{P_K}
      \mathcal A_{p\ell}
      \left(\{F_c^{[s]}F_c^{[s]\top}\}_{c\in\Lambda_\ell}\right)
      =b_{I_p},
      \qquad p=1,\ldots,P_R.
\end{aligned}
\label{eq:three-axis-bm}
\end{equation}
Equation~\eqref{eq:three-axis-bm} exposes the three independent sums that
drive the distributed design.  The Constraint axis partitions the output
rows of $\mathcal A$; the Rank axis partitions the sum forming each
$F_cF_c^\top$; and the Cone axis partitions the block sum in each constraint,
which may couple multiple cones.

The $N_{\rm dev}=P_RP_KP_C$ devices are indexed by $(p,s,\ell)$.  Device
$(p,s,\ell)$ stores the operator shard
$\{A_{c,i}:i\in I_p,\ c\in\Lambda_\ell\}$ and the factor slabs
$\{F_c^{[s]}:c\in\Lambda_\ell\}$.  Operator shards are replicated along the
Rank axis, factors along the Constraint axis, and each objective block $C_c$
along both the Constraint and Rank axes within its cone group.  Row vectors
such as the multiplier and primal residual are replicated along the Rank and
Cone axes.  We write
$\mathsf{AR}_{R}$, $\mathsf{AR}_{K}$, and $\mathsf{AR}_{C}$ for summation by
AllReduce along the corresponding mesh axis; the result remains replicated
along the reduced axis.

\subsection{Distributed evaluation of the ALM--BM operators}
\label{subsec:distributed-operators}

The primal residual, gradient, Hessian action used by the curvature test,
and matrix--vector products used by the dual-slack test are all compositions
of a local forward product and a local adjoint--factor product.
Making these two primitives explicit both shortens the derivation and
clarifies where communication is unavoidable.

For factor-shaped families $U$ and $V$ with the same distribution as $F$,
device $(p,s,\ell)$ forms the local forward partial
\begin{equation}
    \bigl[\Pi_{p s\ell}(U,V)\bigr]_i
    :=\sum_{c\in\Lambda_\ell}
      \left\langle A_{c,i},U_c^{[s]}V_c^{[s]\top}\right\rangle,
    \qquad i\in I_p.
\label{eq:local-forward-partial}
\end{equation}
For a row vector $z_p\in\mathbb R^{|I_p|}$ and a factor-shaped family $U$,
the corresponding local adjoint--factor partial is
\begin{equation}
    \Psi_{c,ps}(z_p;U)
    :=\sum_{i\in I_p}(z_p)_i A_{c,i}U_c^{[s]},
    \qquad c\in\Lambda_\ell.
\label{eq:local-adjoint-factor}
\end{equation}
The first object is row-shaped and is reduced over the Rank and Cone axes; the
second is factor-shaped and is reduced over the Constraint axis.  In
particular,
\begin{equation}
    \mathsf{AR}_{R}\!\left(\Psi_{c,ps}(z_p;U)\right)
    =(\mathcal A^*z)_cU_c^{[s]}.
\label{eq:adjoint-factor-identity}
\end{equation}
This ordering communicates only a thin $n_c\times |K_{c,s}|$ product rather
than an assembled $n_c\times n_c$ slack matrix.

\paragraph{Primal residual.}
Using \eqref{eq:rank-separable-outer-product}, the residual slice owned by
constraint group $p$ is
\begin{equation}
    r_p
    =\mathsf{AR}_{K}\!\left(
       \mathsf{AR}_{C}\!\left(\Pi_{p s\ell}(F,F)\right)
      \right)-b_{I_p}.
\label{eq:distributed-primal-residual}
\end{equation}
There is no reduction over the Constraint axis: every component of $r_p$
belongs to exactly one constraint group.  The shifted multiplier
$\widehat y_p=y_p-\rho r_p$ is therefore formed locally.

\paragraph{Augmented-Lagrangian gradient.}
For $c\in\Lambda_\ell$, the local factor slab of the gradient is
\begin{equation}
\begin{aligned}
    \bigl[\nabla\Phi_{\rho,y}(F)\bigr]_c^{[s]}
    &=2C_cF_c^{[s]}
      -2\mathsf{AR}_{R}\!\left(
          \Psi_{c,ps}(\widehat y_p;F)
        \right)=2S_c(\widehat y)F_c^{[s]}.
\end{aligned}
\label{eq:distributed-gradient}
\end{equation}
Thus the residual and gradient realize the same forward/adjoint duality as
the distributed $Ax/A^\top y$ pattern in D-PDLP, with the additional Cone
reduction induced by block-diagonal semidefinite structure
\cite{DPDLP}.

\paragraph{Curvature test.}
Let $V$ be a spectral-search vector distributed like $F$.  The nonlocal part of the
Hessian action depends on
\begin{equation}
    \xi(V):=\mathcal A(FV^\top+VF^\top).
\label{eq:directional-constraint-image}
\end{equation}
Its constraint-group slice is obtained by the same forward reduction as the
primal residual:
\begin{equation}
    \xi_p(V)=\mathsf{AR}_{K}\!\left(
       \mathsf{AR}_{C}\!\left(
          \Pi_{p s\ell}(F,V)+\Pi_{p s\ell}(V,F)
       \right)
    \right).
\label{eq:distributed-directional-image}
\end{equation}
Define the combined local thin partial
\begin{equation}
    \Theta_{c,ps}(V)
    :=-\Psi_{c,ps}(\widehat y_p;V)
      +\rho\,\Psi_{c,ps}(\xi_p(V);F).
\label{eq:combined-hessian-partial}
\end{equation}
The two adjoint contributions can therefore be accumulated before
communication, giving the local Hessian slab with one Constraint-axis
reduction:
\begin{equation}
    \bigl[H_{\rho,y}(F)[V]\bigr]_c^{[s]}
    =2C_cV_c^{[s]}
     +2\mathsf{AR}_{R}\!\left(\Theta_{c,ps}(V)\right),
    \qquad c\in\Lambda_\ell.
\label{eq:distributed-hessian-vector}
\end{equation}
Consequently, every Hessian--vector product requires one Rank--Cone forward
reduction and one Constraint-axis thin-product reduction; no dense Hessian is
formed.

\paragraph{Dual-slack test.}
When the fixed-rank ALM returns, its final shifted multiplier $\widehat y$
is the returned multiplier $y_{\mathrm{out}}$ in
Algorithm~\ref{alg:fixed-rank-alm}.  The outer rank-adaptation test then
examines the smallest eigenvalue of each relevant slack block
$S_c(y_{\mathrm{out}})=C_c-(\mathcal A^*y_{\mathrm{out}})_c$.  Its
matrix--vector products are evaluated without assembling the slack:
\begin{equation}
    S_c(y_{\mathrm{out}})v
    =C_cv-
      \mathsf{AR}_{R}\!\left(
        \sum_{i\in I_p}(y_{\mathrm{out},p})_iA_{c,i}v
      \right).
\label{eq:distributed-slack-matvec}
\end{equation}
Each cone group tests its local cones, assigning each cone to a single Rank
coordinate $s=\pi_K(c)\in[P_K]$ to avoid redundant work.  Different cone
tests may be assigned to different Rank coordinates.  A negative Rayleigh
quotient provides a valid witness for increasing that block's rank.  The
resulting extremal-eigenvalue estimates and rank-growth flags require only
scalar gathers or reductions among the owning devices.  This formulation is
consistent with the factor-shaped communication in
\eqref{eq:distributed-gradient}; it does not assume that a full dual-slack
matrix has already been assembled.
For each retained subspace $V_c$, the operator evaluations
$\mathcal A_c(V_cZ_cV_c^\top)$ required by the joint rank-lift problem reuse
the forward-reduction pattern in
\eqref{eq:distributed-primal-residual}.  The coefficients of this small
problem are assembled by scalar reductions.  Each resulting correction is
compressed to its nonzero factor columns before those columns are assigned to
Rank shards.

\paragraph{Exactness of the distributed operator actions.}
The partitions in \eqref{eq:constraint-cone-partitions} and
\eqref{eq:rank-partition} are disjoint and exhaustive.  Linearity of
$\mathcal A$ and $\mathcal A^*$ therefore shows that
\eqref{eq:distributed-primal-residual},
\eqref{eq:distributed-gradient},
\eqref{eq:distributed-hessian-vector}, and
\eqref{eq:distributed-slack-matvec} coincide exactly with their single-device
counterparts in exact arithmetic.  In floating-point arithmetic they are
algebraically identical up to the reassociation induced by collective
reductions.  Thus distribution does not change the operator being evaluated.

The remaining inner-solver operations require only small-payload scalar
collectives.  Factor-space quantities such as gradient norms, L-BFGS
curvature-pair products, and spectral-search normalization reduce over the
Rank and Cone axes, whereas residual-space quantities reduce over the
Constraint axis.  The coefficients of the quartic line search combine both
types and are reduced over the axes on which their nonreplicated partial sums
reside.  Each reduction is performed on one replica whenever its input is
replicated, so that the same contribution is not counted $P_R$, $P_K$, or
$P_C$ times.  Their payload is small compared with the thin operator
reductions.

The product-cone decomposition in this section is algebraic: it changes
neither the objective nor the derivatives.  Consequently, its assembled
operators are exactly those used by the theory and algorithm in
Sections~\ref{sec:theoretical-foundations}
and~\ref{sec:algorithm-convergence}.

\subsection{Workload-aware partitioning}
\label{subsec:workload-aware-partitioning}

Equal numbers of constraints, factor columns, or cones rarely imply equal
work: operator nonzeros may be clustered and cone dimensions may vary
widely.  Following the nonzero-aware data-layout principle of
D-PDLP~\cite{DPDLP}, \name{} first reorders constraint rows to improve
column locality and then partitions contiguous row ranges by cumulative
operator nonzeros.  PSD cones remain indivisible.  Large cones and batches of
compatible small cones are assigned to Cone coordinates by a
longest-processing-time heuristic with weights that reflect each cone's
dimension and rank cap.  At initialization, factor columns are divided as
evenly as possible over the Rank coordinates.  These operations change only
data ordering and ownership; the distributed operators remain the exact
rearrangements established above.

The mesh dimensions $(P_R,P_K,P_C)$ are runtime configuration parameters
satisfying $P_RP_KP_C=N_{\rm dev}$.  Increasing one dimension reduces the
corresponding local constraint, factor, or cone work, but also introduces
collectives along that axis.  The appropriate shape therefore depends on the
dominant structure of the instance and on whether its local work is large
enough to amortize collective latency.

%% file: text/kernels.tex

\section{Practical Enhancements}
\label{sec:enhancements}

This section describes the numerical transformations, implementation
safeguards, and structured operator organization used by \name{}.  They do
not modify the ALM--BM control flow of
Section~\ref{sec:algorithm-convergence}, the exact distributed identities of
Section~\ref{sec:distributed-multigpu}, or the stationarity conditions in the
analysis.

\subsection{Block-structured cone operations}
\label{subsec:coneops}

Let $\Ffac=\blkdiag(F_1,\ldots,F_q)$ and write $A_{c,i}$ for the
restriction of the $i$th constraint matrix to cone $c$.  The inner solver
accesses the block SDP through three algebraic primitives:
\begin{align}
  \bigl(\Aop(\Ffac\Ffac^\top)\bigr)_i
    &=\sum_{c=1}^{q}
      \inner{A_{c,i}}{F_cF_c^\top},
  \label{eq:prim-sddmm}\\
  (\Aadj z)_cU_c
    &=\sum_i z_iA_{c,i}U_c,
  \label{eq:prim-spmv}\\
  S_c(z)U_c
    &=C_cU_c-(\Aadj z)_cU_c.
  \label{eq:prim-spmm}
\end{align}
Here $U_c$ may be a factor slab, a Hessian--vector direction, or a slack
Lanczos vector.  Thus \eqref{eq:prim-spmm} covers the gradient, curvature,
and dual-slack actions derived in Section~\ref{sec:distributed-multigpu} without
requiring a dense slack matrix.

For each cone, define the union sparsity pattern
\begin{equation}
  \mathcal S_c
  :=\operatorname{supp}(C_c)
    \cup\bigcup_{i:A_{c,i}\neq0}\operatorname{supp}(A_{c,i}).
  \label{eq:cone-union-pattern}
\end{equation}
Only entries of $F_cF_c^\top$ indexed by $\mathcal S_c$ are required in
\eqref{eq:prim-sddmm}; likewise, the adjoint and slack actions are supported
on $\mathcal S_c$.  The fixed data and union patterns are stored sparsely,
and the adjoint is contracted with $U_c$ before the distributed reduction,
as in Section~\ref{sec:distributed-multigpu}.  Consequently, each pass scales with
the active nonzeros times the relevant factor or vector width, rather than
with the dense cone sizes $n_c^2$.  Neither the primal matrix
$\Ffac\Ffac^\top$, the Hessian, nor a dense global slack is formed or
communicated.

\subsection{Batching small cones}
\label{subsec:conebatch}

Sparse SOS, moment, and chordally decomposed relaxations often contain a few
large cones together with many small cones.  Although the operators in
\eqref{eq:prim-sddmm}--\eqref{eq:prim-spmm} are already sparse, applying them
cone by cone leaves too little work in each small block to amortize fixed
operator overhead.

The direct-sum structure permits an exact aggregation.  For a batch
$\mathcal B$ of cones, define
\begin{equation}
\begin{aligned}
  A_{\mathcal B,i}
:=\blkdiag\bigl(A_{c,i}\bigr)_{c\in\mathcal B},\quad 
  S_{\mathcal B}(z)
    :=\blkdiag\bigl(S_c(z)\bigr)_{c\in\mathcal B},\quad
  F_{\mathcal B}
:=\blkdiag\bigl(F_c\bigr)_{c\in\mathcal B}.
\end{aligned}
\label{eq:batchop}
\end{equation}
Then
\begin{equation}
\begin{aligned}
  \inner{A_{\mathcal B,i}}
        {F_{\mathcal B}F_{\mathcal B}^\top}
    &=\sum_{c\in\mathcal B}\inner{A_{c,i}}{F_cF_c^\top},\\
  S_{\mathcal B}(z)F_{\mathcal B}
    &=\blkdiag\bigl(S_c(z)F_c\bigr)_{c\in\mathcal B}.
\end{aligned}
\label{eq:batch-identities}
\end{equation}
Here $\blkdiag$ denotes a direct sum; it does not imply explicit storage of
the off-diagonal zero blocks.
Batching therefore changes only the granularity at which the same block
operators are evaluated; it does not change the residual, gradient, or
spectral tests.  Cones with similar dimensions, ranks, and sparsity densities
are grouped together, while large cones remain separate.  This distinction
is particularly important when $q$ is large but most $n_c$ are
small.

\subsection{Scaling and preconditioning}
\label{subsec:precond}

The constraint rows and cone blocks can have widely different magnitudes,
which degrades both the ALM penalty update and the conditioning of the
factorized subproblems.  \name{} combines $\ell_\infty$ Ruiz equilibration
with Pock--Chambolle scaling~\cite{PockChambolle2011}, subject to the
requirement that variable scaling preserve the PSD cones.

A generic entrywise scaling of a vectorized symmetric matrix is not, in
general, a congruence and need not preserve positive semidefiniteness.  We
therefore restrict the variable scaling within each cone to a positive
diagonal congruence.  Let $r_i>0$ be a constraint-row scale, let
$D_c\succ0$ be diagonal, and let $\eta_b,\eta_C>0$ be global scaling factors
for the right-hand side and objective.  The primal variables are related by
\begin{equation}
  X_c=\eta_b^{-1}D_c^{-1}\widetilde X_cD_c^{-1},
  \qquad
  F_c=\eta_b^{-1/2}D_c^{-1}\widetilde F_c.
  \label{eq:cone-scaling-map}
\end{equation}
The corresponding scaled data are
\begin{equation}
  \widetilde A_{c,i}
    =r_i^{-1}D_c^{-1}A_{c,i}D_c^{-1},
  \qquad
  \widetilde b_i=\eta_b r_i^{-1}b_i,
  \qquad
  \widetilde C_c=\eta_C D_c^{-1}C_cD_c^{-1}.
  \label{eq:scaled-sdp-data}
\end{equation}
If $\widetilde y$ is the multiplier of the scaled problem, then
\begin{equation}
  y_i=\frac{\widetilde y_i}{r_i\eta_C},
  \qquad
  \widetilde S_c(\widetilde y)
  =\eta_CD_c^{-1}S_c(y)D_c^{-1}.
  \label{eq:scaled-dual-map}
\end{equation}
Because $D_c$ is positive diagonal, these maps preserve PSD order, rank,
primal and dual feasibility, and complementarity.  They also give
\[
  \sum_c\langle\widetilde C_c,\widetilde X_c\rangle
  =\eta_b\eta_C\sum_c\langle C_c,X_c\rangle.
\]
The Ruiz and Pock--Chambolle iterations determine $r_i$ and the diagonal
entries of $D_c$; the default implementation uses this per-coordinate
diagonal mode rather than one scalar per cone.  Scalar nonnegative variables
use the analogous one-dimensional map.

\subsection{Solver defaults}
\label{subsec:defaults}

Let $n_{\mathrm{lp}}$ denote the number of optional scalar nonnegative
components not included in $\sum_c n_c$; algebraically, they may be
treated as $1\times1$ PSD cones.
The settings in Table~\ref{tab:solver-defaults} define the experimental
protocol unless an experiment states otherwise.  We use a comprehensive KKT
stopping test following the DIMACS error framework of the Mittelmann
benchmark~\cite{Mittelmann}.  Let $X=\blkdiag(F_cF_c^\top)_c$, let
$p=\inner{C}{X}$ and $d=b^\top y$, and let
$S(y)=\blkdiag(S_c(y))_c$; optional scalar nonnegative components are included
as $1\times1$ blocks.  Let $\widehat\lambda_c(y)$ denote the smallest
slack eigenvalue returned by the block spectral routine and set
$\widehat\lambda_\oplus(y):=\min_c\widehat\lambda_c(y)$, including the
scalar blocks.  In the original minimization convention, the solver-reported
relative measures are
\begin{equation}
\begin{aligned}
  \mathrm{pres}
    :=\frac{\|\Aop(X)-b\|_2}{1+\|b\|_1}, \qquad 
  \mathrm{dres}
    :=\frac{[-\widehat\lambda_\oplus(y)]_+}{1+\|C\|_1},\qquad 
  \mathrm{gap}
    :=\frac{|p-d|}{1+|p|+|d|},
\end{aligned}
\label{eq:reported-residuals}
\end{equation}
where $[t]_+:=\max\{t,0\}$ and $\|\cdot\|_1$ is the entrywise data norm.
Thus $\mathrm{dres}$ tracks the most negative detected slack direction rather
than the Frobenius distance to the product cone.  The gap compares the primal
minimization value with the candidate dual value, which is a lower bound when
$S(y)\succeq0$.
All quantities in \eqref{eq:reported-residuals}, all objectives, and every
reported slack tolerance are evaluated after applying
\eqref{eq:cone-scaling-map}--\eqref{eq:scaled-dual-map} back to the original,
unscaled problem; internal scaled residuals are not reported.

The spectral routines are implemented matrix-free.  The inner Lanczos
search estimates the smallest eigenpair of the fixed-rank ALM Hessian, while
the outer search estimates a capped set of the most negative Ritz pairs of
each slack block for rank adaptation.  A shifted power iteration supplies the
terminal slack estimate used in the reported dual residual.  None of these
operations forms a dense Hessian or slack matrix.

\begin{table}[H]
\centering
\caption{Default numerical settings and implementation policies used by
\name{}.}
\label{tab:solver-defaults}
\begin{tabular}{p{0.31\linewidth}p{0.61\linewidth}}
\hline
Quantity & Default \\
\hline
Termination
  & Relative primal residual, dual residual, and objective gap at most
  $10^{-4}$\\
Initial BM rank
  & Raw value $\lceil 2\ln m\rceil$, subject to the cone dimension
    and the per-cone cap below \\
Per-cone rank cap
  & $\left\lceil(\sqrt{8m_c+1}-1)/2\right\rceil$, where
    $m_c=|\{i:A_{c,i}\neq0\}|$; used as a heuristic safeguard \\
Inner quasi-Newton solve
  & L-BFGS history size $5$; exact quartic line search on $[0,1]$ for
    L-BFGS steps and on $[-1,1]$ for negative-curvature steps \\
Curvature search
  & At most $15$ matrix-free Hessian--Lanczos steps \\
Rank-growth search
  & At most $30$ Lanczos steps per cone; at most the prescribed
    $r_{\mathrm{inc},c}$ most-negative directions are retained \\
Rank-lift update
  & Small joint rank-lift problem over the retained slack subspaces;
    resulting corrections are compressed before rank sharding \\
Terminal slack estimate
  & Shifted power iteration per cone, used in the reported dual residual \\
ALM penalty
  & $\rho_0=2/\sqrt{\sum_c n_c+n_{\mathrm{lp}}}$; multiplier $3.33$;
    cap $5\times10^5$ \\
Inner-iteration cap
  & $30\,000$ \\
Scaling
  & Ten Ruiz sweeps, Pock--Chambolle scaling with $\alpha=1$, per-coordinate
    positive diagonal PSD congruences, and bound/objective rescaling enabled \\
\hline
\end{tabular}
\end{table}

%% file: text/experiments.tex

\section{Numerical Experiments}
\label{sec:experiments}

The evaluation addresses two complementary questions.  We first assess the
single-GPU solver on the heterogeneous Mittelmann sparse-SDP benchmark, using
both solver-reported termination and the benchmark's uniform DIMACS accuracy
test.  We then examine how the distributed implementation responds to three
structural regimes that motivate the device mesh of
Section~\ref{sec:distributed-multigpu}: vehicle-landing moment relaxations contain many
PSD blocks and stress the Cone axis; electronic-structure SOS relaxations
contain millions of moment-matching constraints and stress the Constraint
axis; and large Max-Cut relaxations produce wide Burer--Monteiro factors and
stress the Rank axis.  These families emphasize, but do not statistically
isolate, the corresponding axes: problem size, sparsity, rank trajectory, and
ALM iteration count can change together.  The experiments therefore
characterize observed end-to-end scaling and identify when computation along
each axis amortizes the associated collectives.

\paragraph{Experimental protocol.}
All \name{} runs use FP64 arithmetic and a wall-clock limit of $3600$\,s.
Unless stated otherwise, the stopping criteria and numerical parameters are
those of Section~\ref{subsec:defaults}; MOSEK and cuLoRADS use their default
parameter configurations.  In the
Mittelmann study, baseline
results are taken from the public leaderboard~\cite{Mittelmann}, and
\name{} is run on an NVIDIA H100 GPU matching the hardware reported for
cuLoRADS~\cite{cuLoRADS}.  All \name{} experiments use one node with two
48-core Intel Xeon Platinum 8468 CPUs and four NVIDIA H100 SXM GPUs, each
with $80$\,GB of memory; the distributed studies use at most four of these
GPUs, interconnected by NVLink/NVSwitch.  In the $(P_R,P_K,P_C)$ convention of
Section~\ref{sec:distributed-multigpu}, the Cone-, Constraint-, and Rank-axis
studies use meshes $(1,1,p)$, $(p,1,1)$, and $(1,p,1)$, respectively.
For $p$ GPUs, we report the observed wall-clock speedup
\begin{equation}
  S_p:=\frac{T_1}{T_p},
  \label{eq:experimental-speedup}
\end{equation}
computed from the underlying timings before display rounding.  The symbol
$f$ means that the target specified for that table was not reached within
$3600$\,s; for a native-termination table, this means that the corresponding
solver did not report successful termination under its own criterion.  The
symbol does not distinguish a timeout from other unsuccessful termination
modes.  For runs using the default stopping rule, native success is determined
by the $10^{-4}$ tolerances on the three solver-reported quantities in
\eqref{eq:reported-residuals}.
Table-specific throughput checkpoints introduced below are reported only as
scaling measurements.
Columns for comparison solvers provide contextual reference points and are
not used to compute $S_p$.

\subsection{Mittelmann sparse-SDP benchmark}
\label{subsec:mittelmann}

We begin with the standard sparse-SDP benchmark of
Mittelmann~\cite{Mittelmann}, which contains $75$ instances spanning
combinatorial optimization, control, and relaxation problems.  The
leaderboard reports three complementary statistics.  \emph{Solved} counts
the instances on which a solver terminates according to its own stopping
rule; because these rules differ, this column is not a uniform accuracy
comparison.  \emph{DIMACS} counts the instances satisfying the benchmark's
common tolerance on the six scaled measures of primal infeasibility, dual
infeasibility, and duality gap.  \emph{SGM10} is the shifted geometric mean
of solution times with a $10$\,s shift and is the leaderboard's aggregate
efficiency metric; lower values are better.

\begin{table}[H]
\centering
\caption{Mittelmann sparse-SDP benchmark ($75$ instances; $3600$\,s limit).
\emph{Solved} uses each solver's native stopping rule, \emph{DIMACS} applies
the benchmark's uniform accuracy test, and lower values of \emph{SGM10} are
better.
\name{} is shaded.}
\label{tab:mittelmann}
\begin{tabular}{lccc}
\toprule
\textbf{Solver} & \textbf{solved/75} & \textbf{DIMACS/75}
& \textbf{SGM10}\\
\midrule
\multicolumn{4}{l}{\textit{GPU solvers}}\\
\rowcolor{gray!15}
\name{}  & 65 & 65 & 113.03\\
cuLoRADS & 68 & 56 & 96.00\\
\midrule
\multicolumn{4}{l}{\textit{CPU solvers}}\\
COPT     & 75 & 70 & 37.69\\
MOSEK    & 66 & 66 & 122.66\\
CSDP     & 65 & 60 & 158.27\\
SDPT3    & 67 & 54 & 156.08\\
SeDuMi   & 52 & 51 & 634.05\\
SDPA     & 60 & 42 & 185.85\\
\bottomrule
\end{tabular}
\end{table}

On SGM10, \name{} ranks third among the listed solvers, behind COPT and
cuLoRADS and ahead of MOSEK, SDPT3, CSDP, SDPA, and SeDuMi.  Its DIMACS
count is $65/75$, the highest among the noncommercial solvers in the table
and nine more than cuLoRADS.

The \emph{Solved} and \emph{DIMACS} columns answer different questions and
should therefore be read together.  For example, cuLoRADS reports $68$
instances solved by its internal criterion, of which $56$ satisfy the
uniform DIMACS test.  For \name{}, both counts are $65$.  This agreement is
consistent with its comprehensive native KKT test, but the table does not make
solver-specific stopping rules directly comparable.  We therefore use
DIMACS for cross-solver accuracy comparisons and SGM10 for aggregate timing.

\subsection{Cone-axis scaling on sparse moment relaxations}
\label{subsec:conewise}

We study Cone-axis parallelism using semidefinite relaxations of the
vehicle-landing optimal-control problem of~\cite{kang2024fast}, which we
abbreviate as VL.
Its sparse Lasserre relaxation organizes consecutive trajectory states and
controls into overlapping cliques.  Each stage contributes one
$190\times190$ moment block and approximately ten $19\times19$ localizing
blocks, while consensus constraints couple adjacent stages.  The resulting
family has $11N-1$ cones at horizon $N$, so increasing $N$ exposes more
indivisible blocks to the Cone-axis assignment.  We generate VL$N$ for
$N\in\{5,10,20,30,50,100,150\}$; the original formulation uses $N=50$.
Table~\ref{tab:conewise} reports the time to each solver's native successful
termination, so its baseline columns are contextual rather than
same-criterion comparisons.

\begin{table}[htbp]
\centering
\caption{Cone-axis scaling on vehicle-landing moment relaxations.  Entries are
times in seconds to each solver's native successful termination; speedups use
the one-GPU time; $f$ denotes failure within $3600$\,s; the fastest \name{}
time in each row is bold.}
\label{tab:conewise}
{\setlength{\tabcolsep}{3.4pt}
\begin{tabular}{rrrrrrrrrrr}
\toprule
& & &
& \textbf{1 GPU} & \multicolumn{2}{c}{\textbf{2 GPUs}}
& \multicolumn{2}{c}{\textbf{4 GPUs}}
& \textbf{MOSEK} & \textbf{cuLoRADS}\\
\cmidrule(lr){5-5}\cmidrule(lr){6-7}\cmidrule(lr){8-9}
\cmidrule(lr){10-10}\cmidrule(lr){11-11}
\textbf{$N$} & \textbf{cones} & \textbf{$n$} & \textbf{$m$}
& \textbf{time} & \textbf{time} & \textbf{speedup}
& \textbf{time} & \textbf{speedup} & \textbf{time} & \textbf{time}\\
\midrule
5   & 54   & 1\,881  & 89\,942     & 9.16    & \textbf{8.99}  & $1.02\times$ & 17.04           & $0.54\times$ & 214.11  & $f$\\
10  & 109  & 3\,781  & 177\,906    & 72.07   & 63.93          & $1.13\times$ & \textbf{45.87}  & $1.57\times$ & 432.87  & $f$\\
20  & 219  & 7\,581  & 353\,823    & 92.17   & 76.65          & $1.20\times$ & \textbf{69.11}  & $1.33\times$ & 800.81  & $f$\\
30  & 329  & 11\,381 & 529\,706    & 145.30  & 106.78         & $1.36\times$ & \textbf{78.76}  & $1.84\times$ & 1294.26 & $f$\\
50  & 549  & 18\,981 & 872\,903    & 211.96  & 156.55         & $1.35\times$ & \textbf{131.96} & $1.61\times$ & 2400.62 & $f$\\
100 & 1099 & 37\,981 & 1\,743\,628 & 537.82  & 342.99         & $1.57\times$ & \textbf{224.14} & $2.40\times$ & $f$     & $f$\\
150 & 1649 & 56\,981 & 2\,614\,431 & 1797.68 & 968.48         & $1.86\times$ & \textbf{686.03} & $2.62\times$ & $f$     & $f$\\
\bottomrule
\end{tabular}
}
\end{table}

Table~\ref{tab:conewise} exhibits the expected crossover.  VL5 is too small
to benefit from four GPUs, whereas four GPUs are the fastest \name{}
configuration from VL10 onward.  Relative to one GPU, the four-GPU speedup
ranges from $1.33\times$ on VL20 to $2.62\times$ on VL150; on the two largest
instances, the runtime falls from $537.82$ to $224.14$\,s and from $1797.68$
to $686.03$\,s, respectively.  Thus Cone-axis parallelism becomes more useful
as the number of assignable blocks grows, while the smallest instance remains
communication-bound.  MOSEK reports native successful termination for
instances up to VL50; neither baseline reaches its own native target on VL100
or VL150.  Because the baseline stopping criteria differ from \name{}'s
native criterion, these columns are contextual timing references rather than
same-criterion speedups.

\subsection{Constraint-axis scaling on electronic-structure SOS relaxations}
\label{subsec:rowwise}

We use the spin-free level-$2$ fermionic SOS relaxation of Low
et al.~\cite[Appendix~G]{low2025fast}.  The exact optimum of this convex
relaxation lower bounds the ground-state energy.  Each $N$-orbital instance
has three PSD cones of orders
\begin{equation}
  \bigl[2N^2,\;N,\;N\bigr],
  \qquad n=2N(N+1),
  \qquad m=\Theta(N^4),
  \label{eq:chemistry-sos-size}
\end{equation}
so the computational burden is dominated by the moment-matching constraints
rather than by the number of cones.  The benchmark suite from~\cite{low2025fast}
contains Fe$_2$S$_2$, Fe$_4$S$_4$, FeMoCo$_{54}$, FeMoCo$_{76}$, and the
P450CPD1X active space.  FeMoCo$_{76}$ is the largest chemistry instance in
that study, where cuLoRADS was used for the SDP computation.

The default solver tolerance is not an energy-accuracy statement.  Moreover,
these instances exhibit a long, high-penalty tail during which conditioning
effects in the outer ALM obscure the scaling of the distributed operators.  We
therefore measure time to the relaxed throughput checkpoint
\begin{equation}
  \mathrm{pres}<10^{-3},
  \qquad
  \mathrm{gap}<10^{-2}.
  \label{eq:chemistry-scaling-checkpoint}
\end{equation}
This checkpoint is deliberately looser than the default stopping test and is
used only to compare Constraint-axis wall-clock scaling, not to assess final
energy accuracy.  Under the same checkpoint and time limit, neither
MOSEK nor cuLoRADS reaches the target in Table~\ref{tab:rowwise}.

\begin{table}[htbp]
\centering
\caption{Constraint-axis scaling on chemistry SOS-SDPs.  Entries are times in
seconds to \eqref{eq:chemistry-scaling-checkpoint}; speedups use the one-GPU
time; $f$ denotes failure within $3600$\,s; the fastest \name{} time in each
row is bold.}
\label{tab:rowwise}
{\setlength{\tabcolsep}{4pt}
\begin{tabular}{lrrrrrrrrrr}
\toprule
& & & &
  \textbf{1 GPU} & \multicolumn{2}{c}{\textbf{2 GPUs}}
  & \multicolumn{2}{c}{\textbf{4 GPUs}}
  & \textbf{MOSEK} & \textbf{cuLoRADS}\\
\cmidrule(lr){5-5}\cmidrule(lr){6-7}\cmidrule(lr){8-9}
\cmidrule(lr){10-10}\cmidrule(lr){11-11}
\textbf{Problem} & \textbf{$N$} & \textbf{$n$} & \textbf{$m$}
& \textbf{time} & \textbf{time} & \textbf{speedup}
& \textbf{time} & \textbf{speedup} & \textbf{time} & \textbf{time}\\
\midrule
Fe$_{2}$S$_{2}$ & 20 & 840     & 160\,K  & 2.63   & 2.48   & $1.06\times$ & \textbf{2.37}  & $1.11\times$ & $f$ & $f$\\
Fe$_{4}$S$_{4}$ & 36 & 2\,664  & 1.68\,M & 11.87  & 5.62   & $2.11\times$ & \textbf{4.93}  & $2.41\times$ & $f$ & $f$\\
FeMoCo$_{54}$   & 54 & 5\,940  & 8.5\,M  & 34.72  & 17.15  & $2.02\times$ & \textbf{8.84}  & $3.93\times$ & $f$ & $f$\\
P450CPD1X       & 58 & 6\,844  & 11.3\,M & 121.45 & 74.82  & $1.62\times$ & \textbf{44.68} & $2.72\times$ & $f$ & $f$\\
FeMoCo$_{76}$   & 76 & 11\,704 & 33\,M   & 318.91 & 131.24 & $2.43\times$ & \textbf{78.36} & $4.07\times$ & $f$ & $f$\\
\bottomrule
\end{tabular}
}
\end{table}

The smallest instance shows little benefit from additional devices.  For the
four larger problems, the observed four-GPU speedup ranges from $2.41\times$
to $4.07\times$, with the largest values on the FeMoCo instances.  P450CPD1X
has a comparable largest block order but a smaller speedup; this difference
may reflect both workload balance and a different ALM or rank trajectory, so
it should not be attributed to row imbalance alone.  Overall, the larger
values of $m$ provide more local operator work with which to amortize the
Constraint-axis thin-product reduction.

\subsection{Rank-axis scaling on Max-Cut relaxations}
\label{subsec:rankwise}

Finally, we study the Goemans--Williamson Max-Cut relaxation
\cite{GoemansWilliamson1995} on graphs from the Stanford SNAP collection
\cite{SNAPdatasets}:
\begin{equation}
  \max_{X\in\Sym{N}}
  \frac14\inner{L(W)}{X}
  \quad\text{s.t.}\quad
  \operatorname{diag}(X)=\mathbf 1,
  \qquad X\succeq0,
  \label{eq:maxcut-experiment-sdp}
\end{equation}
where $L(W)$ is the graph Laplacian.  This is a single-cone SDP with
$n=m=N$; the main dimension available for distribution is the width of the
adaptive factor $F\in\mathbb R^{N\times k}$.  The Rank axis partitions the
columns of $F$, while the corresponding constraint images are combined by the
forward reduction of Section~\ref{sec:distributed-multigpu}.
To match the minimization convention used throughout the paper, the solver is
given
\begin{equation}
  C=-\frac14L(W),
  \qquad \Aop(X)=\operatorname{diag}(X),
  \qquad b=\mathbf 1.
  \label{eq:maxcut-minimization-data}
\end{equation}
Hence the reported primal value $\inner{C}{X}$ is the negative of the
relaxation's maximization value, and the residuals and primal--dual gap use
the definitions in \eqref{eq:reported-residuals} without a sign convention
change.

We construct nested real-world graph instances using $h$-cores
\cite{Seidman1983kcores}: thresholds $h\in\{8,6,5,4,3,2\}$ are applied to
the SNAP \textsf{com-DBLP} graph, and the $2$-core of
\textsf{com-Amazon} supplies a similarly sized instance from a different
domain.  Table~\ref{tab:rankwise} reports the surviving vertex and edge
counts.  For this scaling study, time is measured to the same throughput
checkpoint as in
\eqref{eq:chemistry-scaling-checkpoint},
\[
  \mathrm{pres}<10^{-3},
  \qquad
  \mathrm{gap}<10^{-2}.
\]
\begin{table}[H]
\centering
\caption{Rank-axis scaling on Max-Cut SDP relaxations of SNAP graphs.  Entries
are times in seconds to \eqref{eq:chemistry-scaling-checkpoint}; speedups use
unrounded timings; $f$ denotes failure within $3600$\,s; the fastest \name{}
time in each row is bold.}
\label{tab:rankwise}
{\setlength{\tabcolsep}{4pt}
\begin{tabular}{lrrrrrrrrr}
\toprule
& & &
  \textbf{1 GPU} & \multicolumn{2}{c}{\textbf{2 GPUs}}
  & \multicolumn{2}{c}{\textbf{4 GPUs}}
  & \textbf{MOSEK} & \textbf{cuLoRADS}\\
\cmidrule(lr){4-4}\cmidrule(lr){5-6}\cmidrule(lr){7-8}
\cmidrule(lr){9-9}\cmidrule(lr){10-10}
\textbf{Instance} & \textbf{$N$} & \textbf{$|E|$}
& \textbf{time} & \textbf{time} & \textbf{speedup}
& \textbf{time} & \textbf{speedup} & \textbf{time} & \textbf{time}\\
\midrule
DBLP 8-core   & 33\,875  & 285\,867    & 1.22 & \textbf{0.42} & $2.89\times$ & 0.84 & $1.45\times$ & $f$ & 1.32\\
DBLP 6-core   & 68\,070  & 459\,741    & 1.30 & \textbf{0.73} & $1.79\times$ & 1.12 & $1.16\times$ & $f$ & 1.97\\
DBLP 5-core   & 98\,942  & 584\,570    & 4.53 & 2.70 & $1.68\times$ & \textbf{1.64} & $2.77\times$ & $f$ & 2.76\\
DBLP 4-core   & 143\,129 & 730\,284    & 2.10 & 1.35 & $1.56\times$ & \textbf{0.69} & $3.06\times$ & $f$ & 3.73\\
DBLP 3-core   & 204\,031 & 884\,337    & 3.74 & 1.49 & $2.51\times$ & \textbf{1.01} & $3.69\times$ & $f$ & 4.69\\
DBLP 2-core   & 271\,646 & 1\,004\,432 & 7.11 & 3.80 & $1.87\times$ & \textbf{2.32} & $3.07\times$ & $f$ & 7.79\\
Amazon 2-core & 305\,892 & 896\,901    & 7.36 & 2.90 & $2.54\times$ & \textbf{2.14} & $3.44\times$ & $f$ & 16.63\\
\bottomrule
\end{tabular}
}
\end{table}

The two smallest DBLP instances are latency dominated: two GPUs are faster
than four, and the one-GPU runtimes are only $1$--$2$\,s.  From DBLP 5-core
onward, four GPUs are consistently fastest, with observed speedups between
$2.77\times$ and $3.69\times$.  These ratios should be read as wall-clock
observations rather than precise estimates of parallel efficiency, especially
for the shortest runs and for values slightly above the device count.
Collective reassociation and changes in the adaptive-rank or ALM trajectory
can also make the multi-GPU execution non-work-preserving.  The crossover is
nevertheless consistent with the expected mechanism: increasing graph and
factor work improves the amortization of the Rank-axis reduction.

Across all three families, a material benefit from additional
devices generally appears only after local operator work dominates collective
latency.  Small instances often favor one or two GPUs, whereas larger
instances with many cones, many constraint rows, or wide factors benefit from
the axis aligned with their dominant structural cost.  Thus neither the
device count nor the mesh shape is uniformly best across problem families.

%% file: text/conclusion.tex
\section{Conclusion}
\label{sec:conclusion}

We presented \name{}, a distributed low-rank solver for large-scale,
block-structured SDPs.  At fixed ranks, the method uses matrix-free
L-BFGS--NC iterations to target an approximate Euclidean second-order
stationary point of the factored augmented Lagrangian.  The outer staircase
then tests the shifted dual slack: a negative Rayleigh quotient drives
blockwise rank growth, while satisfaction of the slack tolerance supports
termination.  The reverse multiplier shift makes a negative-slack witness an
exact negative-curvature direction in a newly appended zero column.  The
resulting update admits a closed-form shared amplitude and a joint rank-lift
problem for batches of directions.

Our analysis establishes three distinct conclusions.  Under the
stated fixed-rank ALM and product-manifold assumptions, accumulation points
are second-order critical for the equality-constrained BM problem.  For
almost every product cost, a per-block rank condition then yields global SDP
optimality and finite exact blockwise rank growth.  At finite accuracy,
blockwise slack lower bounds provide an approximate KKT certificate that is
deterministic and verifiable \mbox{a posteriori}.
Under independent full-dimensional perturbations of the cost blocks, a
blockwise tube argument also gives a conditional low-rank
finite-accuracy result and transfers its sampled-cost bounds back to the
nominal costs.  None of these statements supplies a polynomial iteration
bound for the implemented ALM--L-BFGS--NC procedure.

For scalable execution, \name{} distributes constraint rows, factor columns,
and PSD blocks over a Constraint $\times$ Rank $\times$ Cone device mesh.
The distributed residual, gradient, Hessian action, and slack
matrix--vector products are exact rearrangements of their single-device
counterparts in exact arithmetic.  Workload-aware partitioning, batching of
small cones, and PSD-preserving scaling complement this decomposition.  On
the Mittelmann benchmark, \name{} passes the uniform DIMACS test on $65$ of
$75$ instances.  The three multi-GPU studies show an empirical crossover:
additional devices help once the local cone, constraint, or factor work is
large enough to amortize collective latency, whereas small instances can be
faster on fewer GPUs.  The reported chemistry and Max-Cut checkpoint timings
measure throughput rather than final solution accuracy.

\paragraph{Limitations.}  The generic product-cone result excludes a
measure-zero set of block costs, while the finite-accuracy low-rank result
requires independent smoothing of the cost blocks, prespecified rank profiles
and tolerance envelopes, and per-block rank conditions.  It therefore does
not certify an unsmoothed run on an arbitrary fixed cost.  Performance also
depends on ALM conditioning, workload balance, and the
adaptive rank and iteration trajectories, which can differ across device
meshes.  Removing the smoothing requirement and improving topology-aware
balancing and high-penalty preconditioning are natural directions for future
work.

%% file: text/appendix.tex
\section{Additional Proofs}
\label{app:additional-proofs}

\subsection{Product-cone stationarity and landscape}

\begin{proof}[Proof of Proposition~\ref{prop:product-sosp-transfer}]
Let $F\in\mathcal M_{\mathbf k}$ be a feasible full-space second-order
stationary point.  Feasibility gives $\widehat y=y$, and
\eqref{eq:AL-gradient} gives
\begin{equation}
  S_c(y)F_c=0,\qquad c\in[q].
  \label{eq:product-compatible-stationarity-proof}
\end{equation}
Let $\mu_B=\mu_B(F)$ and $\delta:=y-\mu_B$.  The compatibility relation in
\eqref{eq:product-compatible-stationarity-proof} implies
$C_F=\mathcal B_Fy$, and therefore
\[
  \mu_B
  =G(F)^\dagger\mathcal B_F^*C_F
  =G(F)^\dagger G(F)y.
\]
Because $G(F)^\dagger G(F)$ is the orthogonal projector onto
$\operatorname{range}G(F)$,
\[
  \delta\in\ker G(F)=\ker\mathcal B_F,
  \qquad
  \bigl(\mathcal A_c^*(\delta)F_c\bigr)_{c=1}^q=0.
\]
It follows immediately that $S_{B,c}(F)F_c=0$ for every block.

For $U\in T_F\mathcal M_{\mathbf k}$, Assumption~\ref{ass:bm-geometry}
provides a smooth feasible curve
$\Gamma(t)=(\Gamma_c(t))_c$ through $F$ with velocity $U$.  Twice
differentiating
$\langle\delta,c(\Gamma(t))\rangle\equiv0$ and using
$\mathcal B_F\delta=0$ yields
\[
  \sum_{c=1}^q
  \langle U_c,\mathcal A_c^*(\delta)U_c\rangle=0.
\]
Consequently,
\begin{equation}
  \sum_c\langle U_c,S_c(y)U_c\rangle
  =
  \sum_c\langle U_c,S_{B,c}(F)U_c\rangle.
  \label{eq:compatible-slack-transfer}
\end{equation}
The penalty term in \eqref{eq:AL-Hessian-form} vanishes because
$Dc(F)[U]=0$.  Full-space second-order stationarity and
\eqref{eq:compatible-slack-transfer} therefore give the curvature inequality
in \eqref{eq:product-bm-sosp}.  Together with block stationarity, this proves
the proposition.
\end{proof}

\begin{proof}[Proof of Theorem~\ref{thm:product-generic-landscape}]
For each block with $k_c<n_c$, let
\[
  \mathcal V_{c,k_c}
  :=\{H\in\mathbb S^{n_c}:
          \operatorname{rank}(H)\leq n_c-k_c\}.
\]
This determinantal variety has codimension $\tau(k_c)$ in
$\mathbb S^{n_c}$.  Since
$\dim\operatorname{Im}(\mathcal A_c^*)=r_c$,
\[
  \dim\bigl(
    \mathcal V_{c,k_c}+\operatorname{Im}(\mathcal A_c^*)
  \bigr)
  \leq \tau(n_c)-\tau(k_c)+r_c
  <\tau(n_c).
\]
It is therefore a measure-zero subset of $\mathbb S^{n_c}$.

Suppose a first-order critical factor has a full-column-rank block $F_c$
with $k_c<n_c$, and set $\lambda:=\mu_B(F)$.  Stationarity
$S_c(\lambda)F_c=0$ implies
$\operatorname{rank}S_c(\lambda)\leq n_c-k_c$, and hence
\[
  C_c
  =S_c(\lambda)+\mathcal A_c^*(\lambda)
  \in
  \mathcal V_{c,k_c}+\operatorname{Im}(\mathcal A_c^*).
\]
The union of the corresponding cylinders in
$\prod_c\mathbb S^{n_c}$ is measure zero.  Outside that union, every block
with $k_c<n_c$ is column-rank deficient.

Now let $F$ be second-order critical.  If $F_c$ is column-rank deficient,
choose $0\neq z_c\in\ker F_c$.  For arbitrary
$u_c\in\mathbb R^{n_c}$, let $U$ be zero outside block $c$ and set
$U_c=u_cz_c^\top$.  Then $Dc(F)[U]=0$, and
\eqref{eq:product-bm-sosp} gives
\[
  0\leq
  \langle U_c,S_{B,c}(F)U_c\rangle
  =\|z_c\|_2^2u_c^\top S_{B,c}(F)u_c.
\]
Thus $S_{B,c}(F)\succeq0$.  Otherwise $F_c$ has full column rank.  The
generic argument excludes this case when $k_c<n_c$, and $k_c>n_c$ cannot
have full column rank.  Hence $k_c=n_c$, $F_c$ is invertible, and
$S_{B,c}(F)F_c=0$ gives $S_{B,c}(F)=0$.  Every slack block is therefore
positive semidefinite.  Primal feasibility and complementarity complete the
product-SDP KKT conditions.
\end{proof}

\subsection{Finite-accuracy certificates and smoothing}

\begin{proof}[Proof of Proposition~\ref{prop:approx-kkt-gap}]
Let $r:=\sum_c\mathcal A_c(X_c)-b$.  Since
$C_c=\mathcal A_c^*(\lambda)+S_c(\lambda)$,
\[
  \sum_c\langle C_c,X_c\rangle
  =\langle\lambda,b+r\rangle
   +\sum_c\operatorname{tr}(S_c(\lambda)X_c).
\]
Cauchy--Schwarz across the block traces gives
\[
  \left|\sum_c\operatorname{tr}(S_c(\lambda)X_c)\right|
  \leq
  \sqrt{\sum_c n_c}
  \left(\sum_c\|S_c(\lambda)X_c\|_F^2\right)^{1/2}.
\]
Hence
\[
  \sum_c\langle C_c,X_c\rangle
  \leq
  b^\top\lambda
  +\varepsilon_0\|\lambda\|_2
  +\sqrt{\sum_c n_c}\,\varepsilon_1.
\]
For any feasible comparison tuple $(X_c^\star)_c$,
\[
  b^\top\lambda
  =\sum_c\langle C_c,X_c^\star\rangle
   -\sum_c\langle S_c(\lambda),X_c^\star\rangle
  \leq
  \sum_c\langle C_c,X_c^\star\rangle
  +\varepsilon_2\sum_c\operatorname{tr}(X_c^\star).
\]
Combining these inequalities proves
\eqref{eq:approx-kkt-objective-gap}.
\end{proof}

\begin{proof}[Proof of Proposition~\ref{prop:alm-to-afac}]
Set $\lambda=\widehat y(F;y,\rho)$.  Equation~\eqref{eq:AL-gradient} gives
\[
  \left(\sum_c\|S_c(\lambda)F_c\|_F^2\right)^{1/2}
  \leq\frac{\eta}{2}.
\]
For an admissible unit direction in
\eqref{eq:product-afac-second},
\[
  \|Dc(F)[U]\|_2
  =2\left\|\sum_c\mathcal A_c(U_cF_c^\top)\right\|_2
  \leq2\gamma.
\]
The approximate Hessian bound and \eqref{eq:AL-Hessian-form} then imply
\[
  2\sum_c\langle U_c,S_c(\lambda)U_c\rangle
  \geq-\zeta-4\rho\gamma^2.
\]
This is \eqref{eq:product-afac-second} with the tolerances in
\eqref{eq:AL-to-AFAC-tolerances}.
\end{proof}

\begin{proof}[Proof of Theorem~\ref{thm:product-smoothed-low-rank}]
The primal residual bound is part of
\eqref{eq:product-afac-first}.  Moreover,
\begin{align*}
  \left(\sum_c\|S_c(\lambda)X_c\|_F^2\right)^{1/2}
  &\leq
  \left(\sum_c
    \|S_c(\lambda)F_c\|_F^2\|F_c\|_{\mathrm{op}}^2
  \right)^{1/2}\\
  &\leq R_F\epsilon_1.
\end{align*}
It remains to prove approximate dual feasibility.

If $k_c>n_c$, choose a unit $z_c\in\ker F_c$.  For any unit
$u_c\in\mathbb R^{n_c}$, the product direction supported on block $c$ with
$U_c=u_cz_c^\top$ satisfies
$\mathcal A_c(U_cF_c^\top)=0$.  Equation
\eqref{eq:product-afac-second} then gives
$u_c^\top S_c(\lambda)u_c\geq-\epsilon_2$.  The same conclusion holds when
$a_c=0$, now using any unit $z_c\in\mathbb R^{k_c}$.  Hence these blocks
are deterministically approximately dual feasible.

Fix $c\in\mathcal J$.  If
$\sigma_{\min}(F_c)\leq\gamma/a_c$, choose a corresponding unit right
singular vector $z_c$ and again take $U_c=u_cz_c^\top$.  Then
\[
  \|\mathcal A_c(U_cF_c^\top)\|_2
  \leq a_c\|U_cF_c^\top\|_F
  =a_c\|F_cz_c\|_2
  \leq\gamma,
\]
so \eqref{eq:product-afac-second} implies
$S_c(\lambda)\succeq-\epsilon_2I$.  Any block violating approximate dual
feasibility must therefore satisfy
\begin{equation}
  \sigma_{\min}(F_c)>\frac{\gamma}{a_c}.
  \label{eq:product-bad-block-smin}
\end{equation}

Let $\mathcal V_{c,k_c}$ be the determinantal variety used in the proof of
Theorem~\ref{thm:product-generic-landscape}.  The min--max characterization
of singular values and \eqref{eq:product-bad-block-smin} give
\[
  \operatorname{dist}\bigl(S_c(\lambda),\mathcal V_{c,k_c}\bigr)
  \leq
  \frac{\|S_c(\lambda)F_c\|_F}
       {\sigma_{\min}(F_c)}
  <\frac{\epsilon_1a_c}{\gamma}
  =\delta_c.
\]
Thus a bad block can occur only if
\begin{equation}
  C_c\in
  \operatorname{tube}_{\delta_c}(\mathcal V_{c,k_c})
  +\mathcal A_c^*(B_{R_\lambda}),
  \label{eq:product-bad-cost-tube}
\end{equation}
where $B_{R_\lambda}$ is the multiplier ball.

The set $\mathcal A_c^*(B_{R_\lambda})$ lies in an $r_c$-dimensional
subspace and in the Frobenius ball of radius $\kappa_c$.  Since
$\delta_c\leq\kappa_c$, it has a $\delta_c$-net with at most
$(3\kappa_c/\delta_c)^{r_c}$ points.  For each net point,
\eqref{eq:product-bad-cost-tube} is contained in a
$2\delta_c$-tube around a translate of $\mathcal V_{c,k_c}$.
That variety has codimension $\tau(k_c)$ and is defined by minors of degree
$n_c-k_c+1$.  Applying the tube-volume estimate used in
\cite[Theorem~6]{cifuentes2022polynomial}, together with
$\delta_c<\sigma_c/(4n_c^3)$, bounds the probability of
\eqref{eq:product-bad-cost-tube} by
\[
  4\mathrm e
  \left(\frac{3\kappa_c}{\delta_c}\right)^{r_c}
  \left(\frac{4n_c^3\delta_c}{\sigma_c}\right)^{\tau(k_c)}
  =p_c.
\]
A union bound over $c\in\mathcal J$ proves
\eqref{eq:product-smoothed-kkt}.  Substituting
\eqref{eq:product-smoothed-simple-conditions} into
\eqref{eq:product-smoothed-failure} gives the simplified bound stated after
the theorem, by the same algebra as
\cite[Corollary~1]{cifuentes2022polynomial}.
\end{proof}

\subsection{CARDAL convergence and finite outputs}

\begin{proof}[Proof of Theorem~\ref{thm:product-fixed-rank-sosp}]
Lemma~\ref{lem:vanishing-residual} gives
$\|c(F_\ell)\|_2\to0$.  Let
$F_{\ell_j+1}\to F_\star$ and, after taking a further subsequence, let
$y_{\ell_j+1}\to y_\star$.  The residual and gradient bounds imply
\[
  c(F_\star)=0,
  \qquad
  S_c(y_\star)F_{\star,c}=0\quad(c\in[q]).
\]

The product tangent projector is
\[
  \mathcal P_F
  =I-\mathcal B_FG(F)^\dagger\mathcal B_F^*.
\]
Assumption~\ref{ass:bm-geometry} makes $\mathcal P_F$ smooth near
$\mathcal M_{\mathbf k}$.  Given
$U_\star\in\ker Dc(F_\star)$, set
$U_j=\mathcal P_{F_{\ell_j+1}}U_\star$.  Then
$U_j\to U_\star$ and $Dc(F_{\ell_j+1})[U_j]=0$.  Since
$y_{\ell_j+1}$ is the shifted multiplier of the subproblem indexed by
$\ell_j$, the
inner Hessian bound and \eqref{eq:AL-Hessian-form} give
\[
  2\sum_c
  \langle U_{j,c},S_c(y_{\ell_j+1})U_{j,c}\rangle
  \geq-\zeta_{\ell_j}\|U_j\|_\oplus^2.
\]
Passing to the limit and applying the compatible-multiplier argument from
the proof of Proposition~\ref{prop:product-sosp-transfer} yields
\eqref{eq:product-bm-sosp}.

Under the additional rank and generic-cost hypotheses,
Theorem~\ref{thm:product-generic-landscape} makes every accumulation point
globally optimal.  If objective convergence failed, compactness and
continuity would produce an accumulation point whose objective differs from
$p_\star$.  If convergence in distance to $\mathcal X^\star$ failed, the
same argument would produce a represented matrix tuple outside the closed
solution set.  Both contradict global optimality of every accumulation
point.
\end{proof}

\begin{proof}[Proof of Corollary~\ref{cor:product-exact-rank-growth}]
There are finitely many block profiles between
$\mathbf k_0$ and $\mathbf k^{\mathrm{gen}}$, so the union of their
measure-zero exceptional cost sets remains measure zero.

Consider a stage at which block $c$ has reached
$k_c^{\mathrm{gen}}$.  Exact feasibility and first-order stationarity give
$S_c(y)F_c=0$.  If $F_c$ has full column rank, the generic block-deficiency
argument rules out $k_c^{\mathrm{gen}}<n_c$, while
$k_c^{\mathrm{gen}}>n_c$ cannot have full column rank.  Thus
$k_c^{\mathrm{gen}}=n_c$, $F_c$ is invertible, and $S_c(y)=0$.
Otherwise choose a nonzero $z_c\in\ker F_c$.  With $U$ supported only on
block $c$ and $U_c=u_cz_c^\top$, we have $Dc(F)[U]=0$.  Full-space
second-order stationarity and \eqref{eq:AL-Hessian-form} give
\[
  0\leq
  2\|z_c\|_2^2u_c^\top S_c(y)u_c
  \qquad\text{for every }u_c,
\]
so again $S_c(y)\succeq0$.  The exact slack test cannot enlarge this block
again.

Every accepted update increases at least one block that has not reached its
certifying rank.  Summing the maximum number of updates of each block gives
\eqref{eq:product-rank-update-bound}.  At termination, every slack block is
positive semidefinite, while exact feasibility and stationarity give
complementarity.  The primal--dual KKT conditions prove global optimality.
\end{proof}

\begin{proof}[Proof of Proposition~\ref{prop:finite-output-kkt}]
Equation~\eqref{eq:AL-gradient} gives
\[
  \left(\sum_c
    \|S_c(y_{\mathrm{out}})F_c\|_F^2
  \right)^{1/2}
  \leq\frac{\omega}{2}.
\]
Consequently,
\begin{align*}
  \left(\sum_c
    \|S_c(y_{\mathrm{out}})X_c\|_F^2
  \right)^{1/2}
  &\leq
  \left(\sum_c
    \|S_c(y_{\mathrm{out}})F_c\|_F^2\|F_c\|_{\mathrm{op}}^2
  \right)^{1/2}\\
  &\leq\frac{R_F\omega}{2}.
\end{align*}
Primal positive semidefiniteness, the primal-residual bound, and the blockwise
slack bounds are assumed directly, which proves
\eqref{eq:finite-kkt-tolerances}.
\end{proof}

\begin{proof}[Proof of Corollary~\ref{cor:product-smoothed-output}]
Proposition~\ref{prop:alm-to-afac} shows that
$(F,\widehat y)$ is a product $(\epsilon,\gamma)$-AFAC pair with the fixed
tolerances in \eqref{eq:fixed-smoothed-tolerances}.
Theorem~\ref{thm:product-smoothed-low-rank} gives the sampled-cost
conclusion.

Write $C_c=\bar C_c+W_c$.  Since
$\|W_c\|_{\mathrm{op}}\leq\|W_c\|_F\leq\sigma_c$,
\[
  S_{\bar C_c}(\widehat y)
  \succeq-(\epsilon_2+\sigma_{\max})I.
\]
The direct-sum triangle inequality,
$\|X_c\|_F\leq\|F_c\|_F^2$, and
$\sum_c\|F_c\|_F^2\leq R_F^2$ also give
\[
  \left(\sum_c
    \|S_{\bar C_c}(\widehat y)X_c\|_F^2
  \right)^{1/2}
  \leq R_F\epsilon_1+\sigma_{\max}R_F^2.
\]
The affine constraints are unchanged, proving
\eqref{eq:product-nominal-kkt}.  Finally, the sampled bound
$S_{C_c}(\widehat y)\succeq-\epsilon_2I$ excludes a Rayleigh quotient below
$-\tau_{\mathrm{dual}}$ whenever
$\tau_{\mathrm{dual}}\geq\epsilon_2$.
\end{proof}